\documentclass{aptpub}

\authornames{Cornean, Herbst, M\o ller, St\o ttrup, S\o rensen} 
\shorttitle{Singular distribution functions} 

\usepackage{xcolor}
\usepackage{graphicx}
\usepackage{mathtools}
\usepackage{cancel}
\usepackage{subcaption} 
\usepackage{floatrow} 

\DeclareFloatVCode{largevskip}{\vskip 5pt}
\DeclareFloatSeparators{myfill}{\hskip.025\textwidth plus1fill}

\newcommand{\clr}{\color{black}} 

\newcommand{\rmP}{\mathrm{P}}

\makeatletter
\renewcommand\subsubsection{\@startsection{subsubsection}{3}{\z@}%
                       {-18\p@ \@plus -4\p@ \@minus -4\p@}%
                       {4\p@ \@plus 2\p@ \@minus 2\p@}%
                       {\normalfont\normalsize\bfseries\boldmath
                        \rightskip=\z@ \@plus 8em\pretolerance=10000 }}
\makeatletter

\numberwithin{equation}{section}  

\begin{document}

{\clr \title{Singular distribution functions for random\\ variables with stationary digits}} 

\authorone[Aalborg University]{Horia Cornean} 
\authortwo[University of Virginia]{Ira W. Herbst}
\authorone[Aalborg University]{Jesper M{\O}ller}
\authorone[Aalborg University]{Benjamin B. St{\O}ttrup}
\authorone[Aalborg University]{Kasper S. S{\O}rensen}

\addressone{Department of Mathematical Sciences, Aalborg University, Skjernvej 4A, 9220 Aalborg, Denmark} 
\addresstwo{Department of Mathematics, University of Virginia, Charlottesville, VA 22903, USA}

\begin{abstract}
Let $F$ be the cumulative distribution function (CDF) of the base-$q$ expansion $\sum_{n=1}^\infty X_n q^{-n}$, where $q\ge2$ is an integer and $\{X_n\}_{n\geq 1}$ is a stationary stochastic process with state space $\{0,\ldots,q-1\}$.
In a previous paper we characterized the absolutely continuous and the discrete components of $F$.  In this paper we study special cases of models, including 
stationary Markov chains of any order and stationary renewal point processes, where we establish a law of pure types: $F$ is then either a uniform or a singular CDF on $[0,1]$. Moreover, we study   
 mixtures of such models. In most cases expressions and plots of $F$ {\clr are given}.
\end{abstract}



\section{Introduction}\label{s:intro}

A function $F\colon \mathbb{R}\mapsto \mathbb{R}$ for which $F'=0$ Lebesgue almost everywhere on $\mathbb{R}$ is called \emph{singular} (that is, $F'(x)$ exists and is 0 for all $x\in \mathbb{R}\setminus N$ where $N$ is a Lebesgue nullset). For a detailed account of the early history of singular functions, see \cite{Paradis2011}, \cite{Amo2012}, and the references therein. The first and most well-known example of a singular function is the Cantor function \cite{Cantor1884,DOVGOSHEY2006}. Other well-known examples include Minkowski's question-mark function~\cite{Minkowski1904,Denjoy1932,Denjoy1934} and the Riesz-Nagy functions~\cite{Riesz1955}. The latter functions have been treated numerous times in the literature, also before their appearance in \cite{Riesz1955}, 
see e.g.\ \cite{Salem1943}, \cite{Takacs1978}, and \cite{Paradis2007}. 
In more recent times many new constructions of singular functions and generalizations of the well-known examples listed above have appeared in the literature~\cite{Paradis2007,Paradis2011,Amo2012,Okamoto2007,Kairies1997,Wen1998,Sanchez2012,Sanchez2014,Sanchez2016}. 

In a probabilistic setting singular functions are often constructed as follows~\cite{Billingsley1995,Harris1955,Dym1968} (an exception is the paper~\cite{Wen1998}). Let $q\geq 2$ be an integer and $\{X_n\}_{n\geq 1}$ a stochastic process with state space $\{0,\dots,q-1\}$. Define a stochastic variable $X$ on $[0,1]$ by the following base-$q$ expansion with digits $X_1,X_2,\ldots$:
\begin{equation*}
    X\coloneqq (0.X_1 X_2\dots )_q\coloneqq\sum_{n=1}^\infty X_n q^{-n}.
\end{equation*}
Throughout this paper,
\[ 
    F(x)\coloneqq\rmP(X\leq x), \qquad x\in \mathbb{R},
\] 
is the CDF 
of $X$. This is a monotone function, so $F'$ exists almost everywhere and as we shall see in many cases 
$F$ is singular. 
The simplest situation is to assume that $\{X_n\}_{n\geq 1}$ is a Bernoulli scheme, i.e., the $X_n$'s are independent and identically distributed (IID). Then, for $q=2$ and $p\coloneqq\rmP(X_n=0)$, $F$ is a Riesz-Nagy function \cite{Billingsley1995,Takacs1978}, which is the uniform CDF on $[0,1]$ if $p=\tfrac12$, singular continuous if $p\not\in\{0,\tfrac12,1\}$, and concentrated at 0 or 1 if $p=0$ or $p=1$, respectively. 
If instead $q=3$, $\rmP(X_n=1)=0$, and $\rmP(X_n=0)=\rmP(X_n=2)=\tfrac12$, then $F$ is the Cantor function \cite{Billingsley1995}. 

\subsection{Stationary digits}\label{s:setting}

In \cite{part1} we considered the case where $\{X_n\}_{n\geq 1}$ is \emph{stationary}, i.e., $\{X_n\}_{n\ge1}$ and $\{X_n\}_{n\ge2}$ are identically distributed. For short \emph{stationarity} refers to this setting. Below we summarize some characterisation results for stationarity which motivate the objective of the present paper.
For this we 
recall that for any $x\in[0,1]$, there exists a sequence $\{x_n\}_{n\ge1} \subseteq \{0,\dots,q-1\}$ such that $x$ is given by the base-$q$ expansion $x=(0.x_1x_2\ldots)_q$. This expansion is unique except when $x$ is a \emph{base-$q$ fraction} (in $(0,1)$), i.e., when there exist $n\in\mathbb N$ and $x_1,\dots,x_n\in \{0,\dots,q-1\}$ such that $x_n>0$ and $x=(0.x_1\dots x_n00\dots)_q$.
This ambiguity 
plays no role when $\{X_n\}_{n\geq 1}$ is stationary, since then $X$ is almost surely not a base-$q$ fraction (see also Remark~2.2). 

Theorem~1 in \cite{part1} established that stationarity 
is equivalent to that for all base-$q$ fractions $x\in(0,1)$, the functional equation
\begin{equation}\label{e:3}
F(x) = F(0)+\sum_{j=0}^{q-1} [F((x+j)/q)-F(j/q)]
\end{equation}
is satisfied.
Theorem~3 in \cite{part1} showed that stationarity is equivalent to that 
$F$ is a mixture of three CDFs $F_1,F_2,F_3$  whose corresponding probability distributions are mutually singular measures concentrated on $[0,1]$ and so that $F_1,F_2,F_3$ satisfy 
the following statements {\rm (I)-(III)}: 
\begin{enumerate} 
\item[{\rm (I)}]
$F_1$ is the uniform CDF on $[0,1]$, that is, $F_1(x)=x$ for $x\in[0,1]$. 
\item[{\rm (II)}] $F_2$ is a mixture 
of an at most countable number of CDFs of the form
	\begin{align}\label{e:F2type}
	F_{s_1,\ldots,s_k}(x)\coloneqq\frac{1}{k}\sum_{j=1}^{k} H(x-s_j),\qquad x\in\mathbb R,
	\end{align}
	where $H$ denotes the Heaviside function, $s_1,\dots,s_k\in [0,1]$ are pairwise distinct numbers such that $qs_{j+1}-s_j\in\{0,\dots,q-1\}$ for $j=1,\dots,k$, and $s_{k+1}\coloneqq s_1$.
\item[{\rm (III)}]
 $F_3$ is singular continuous and satisfies \eqref{e:3}.
 \end{enumerate}

It is easily verified that $F=F_1$ 
if and only if the $X_n$'s are IID and uniformly distributed on $\{0,\dots,q-1\}$. Note that 
the discrete part $F_2$ is highly constrained as $F_{s_1,\dots,s_k}$ in \eqref{e:F2type} can be obtained from
a $(k-1)$'th order Markov chain which
effectively corresponds to a uniform distribution on a state space consisting of $k$ states, cf.\ Corollary~2.1 in \cite{part1}. So the non-trivial part above is (III), and 
 in many interesting cases of stationary stochastic processes $\{X_n\}_{n\geq 1}$,  as we shall see, $F=F_3$.  

\subsection{Our contribution} 

The main interest in the present paper is to study the properties of $F$ under various classes of stationary models for $\{X_n\}_{n\ge1}$, and in particular to 
establish a better understanding of 
singular continuous CDFs which satisfy \eqref{e:3}.
Figure~1 shows plots of $F$ to illustrate how different it can look depending on which parameter values we choose for different parametric models of $\{X_n\}_{n\ge1}$ 
(Examples~\ref{ex:3}, \ref{ex:symmetry}, \ref{ex:6}, \ref{ex:4}, and \ref{ex:MCbeta} below, where it is explained which curves correspond to which parameter choices for the models).
For all cases of $F$ in Figure~1, $F$ is strictly increasing on $[0,1]$ and apart from 
 the uniform CDF on $[0,1]$ (the straight line in (a)), $F$ is singular continuous on $[0,1]$.

\thisfloatsetup{subfloatrowsep=myfill,captionskip=0pt,rowpostcode =largevskip}%
  \begin{figure}[!htbp]
    \captionsetup[subfigure]{justification=centering}
      \ffigbox[\textwidth]{%
  \begin{subfloatrow}[2]%
    \ffigbox[0.45\textwidth]{\caption{}\label{fig:a}}{%
      \includegraphics[width=0.95\linewidth]{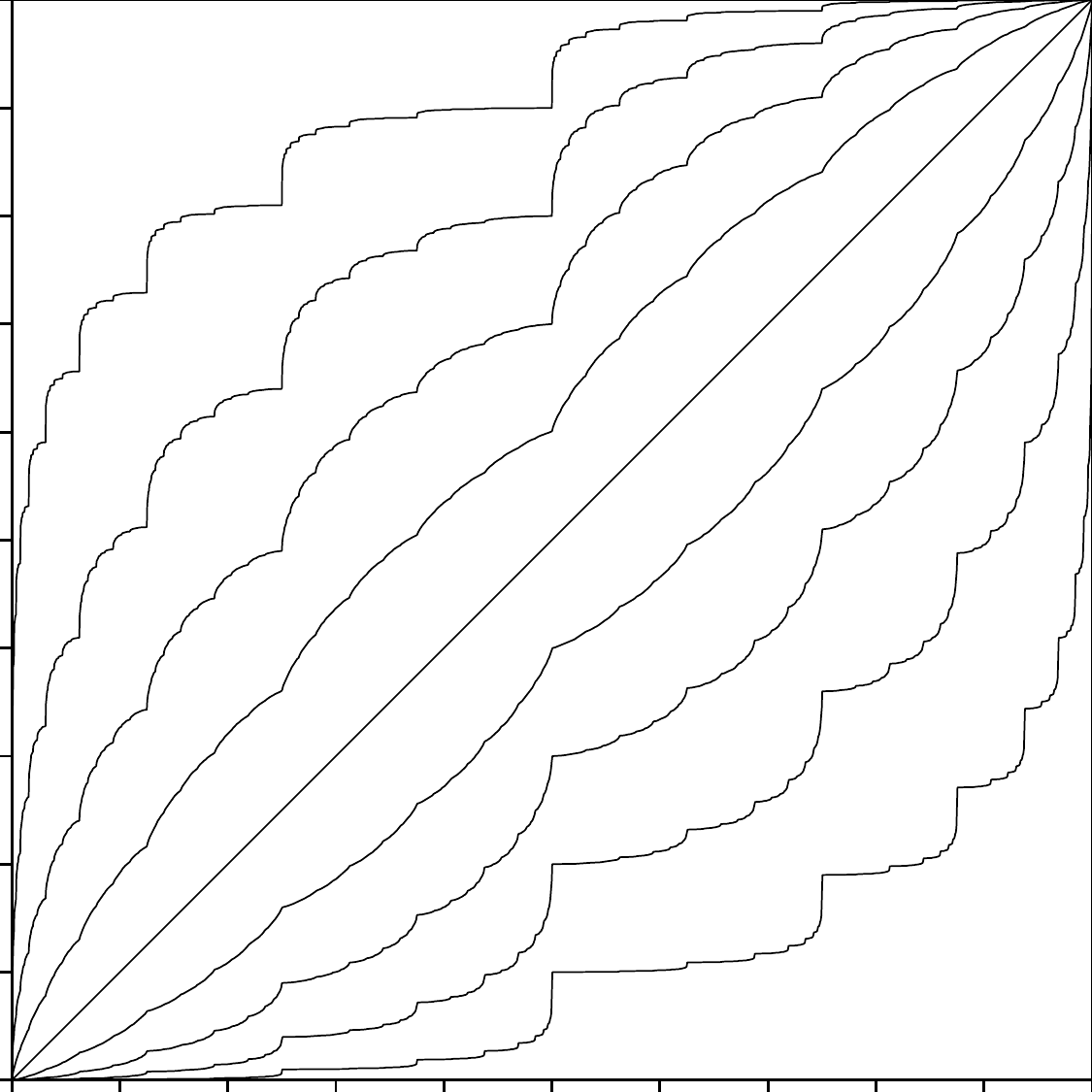}}
    \ffigbox[0.45\textwidth]{ \caption{}\label{fig:b}}{%
      \includegraphics[width=0.95\linewidth]{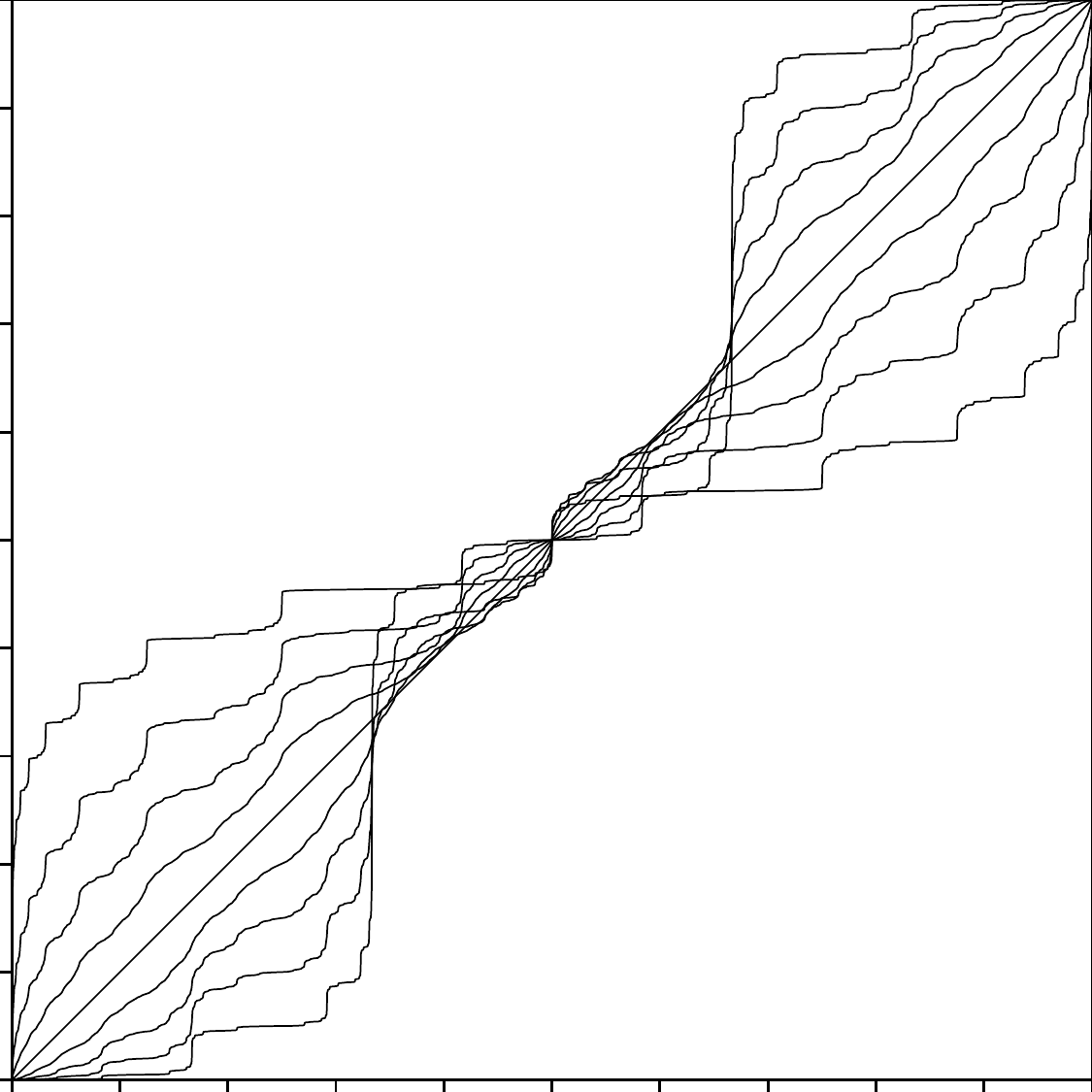}}
  \end{subfloatrow}
    \begin{subfloatrow}[2]%
    \ffigbox[0.45\textwidth]{\caption{}\label{fig:c}}{%
      \includegraphics[width=0.95\linewidth]{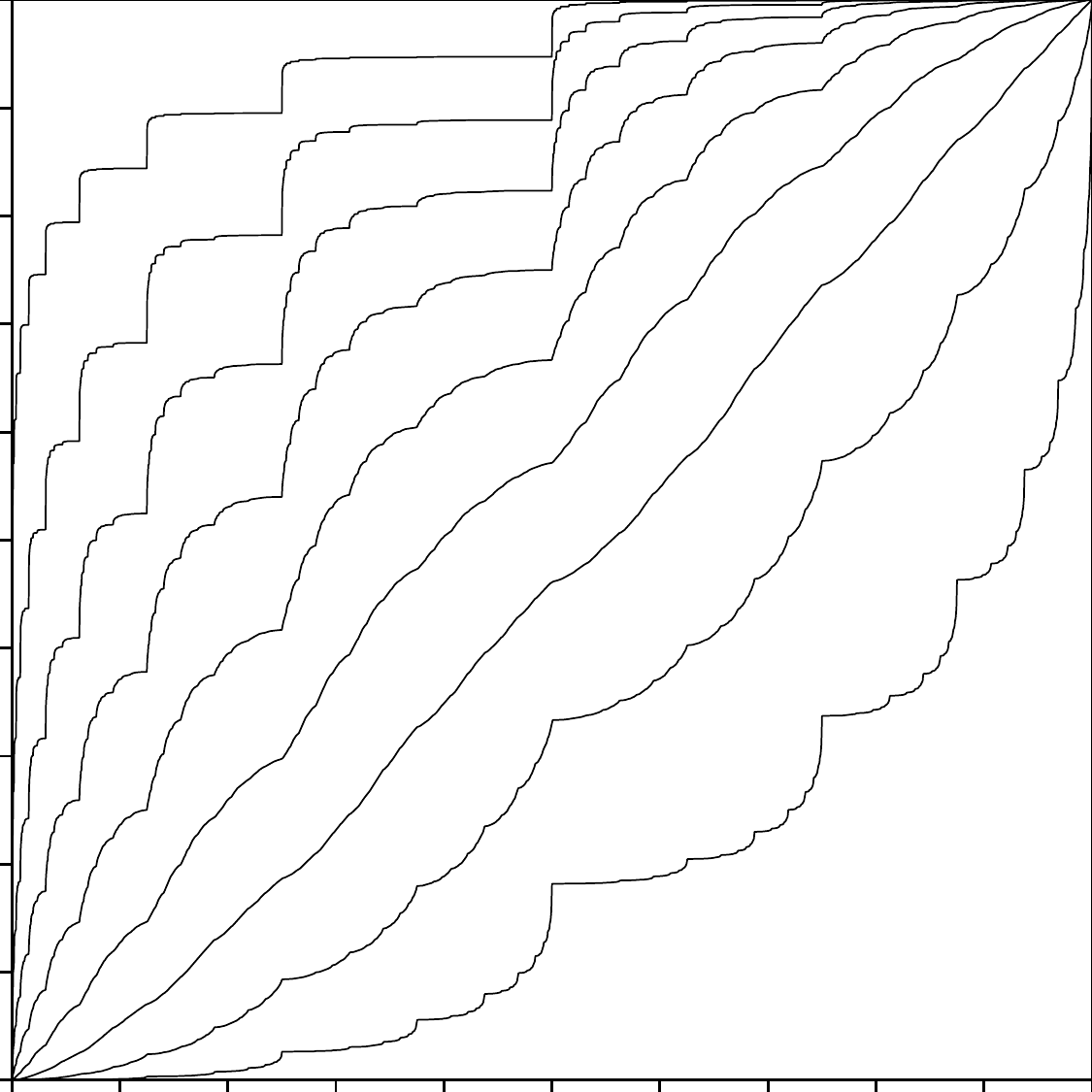}}
    \ffigbox[0.45\textwidth]{ \caption{}\label{fig:d}}{%
      \includegraphics[width=0.95\linewidth]{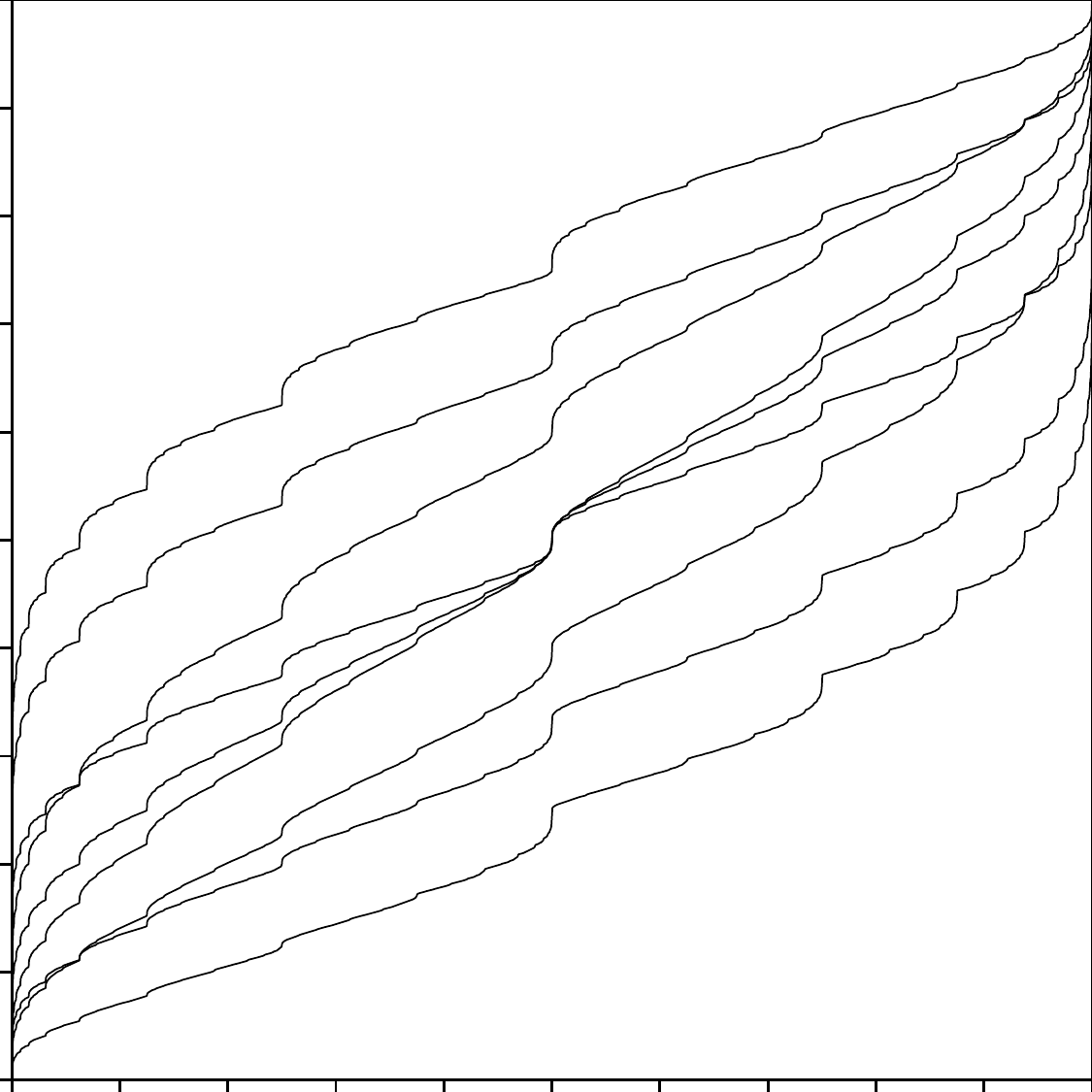}}
  \end{subfloatrow}
  
    \renewlengthtocommand\settowidth\Mylen{\subfloatrowsep}\vskip\Mylen
    \TopFloatBoxes\floatsetup[subfigure]{heightadjust=none}
    \begin{subfloatrow*}
    \ffigbox[0.45\textwidth][][t]{\subcaption{}}{\includegraphics[width=0.95\linewidth]{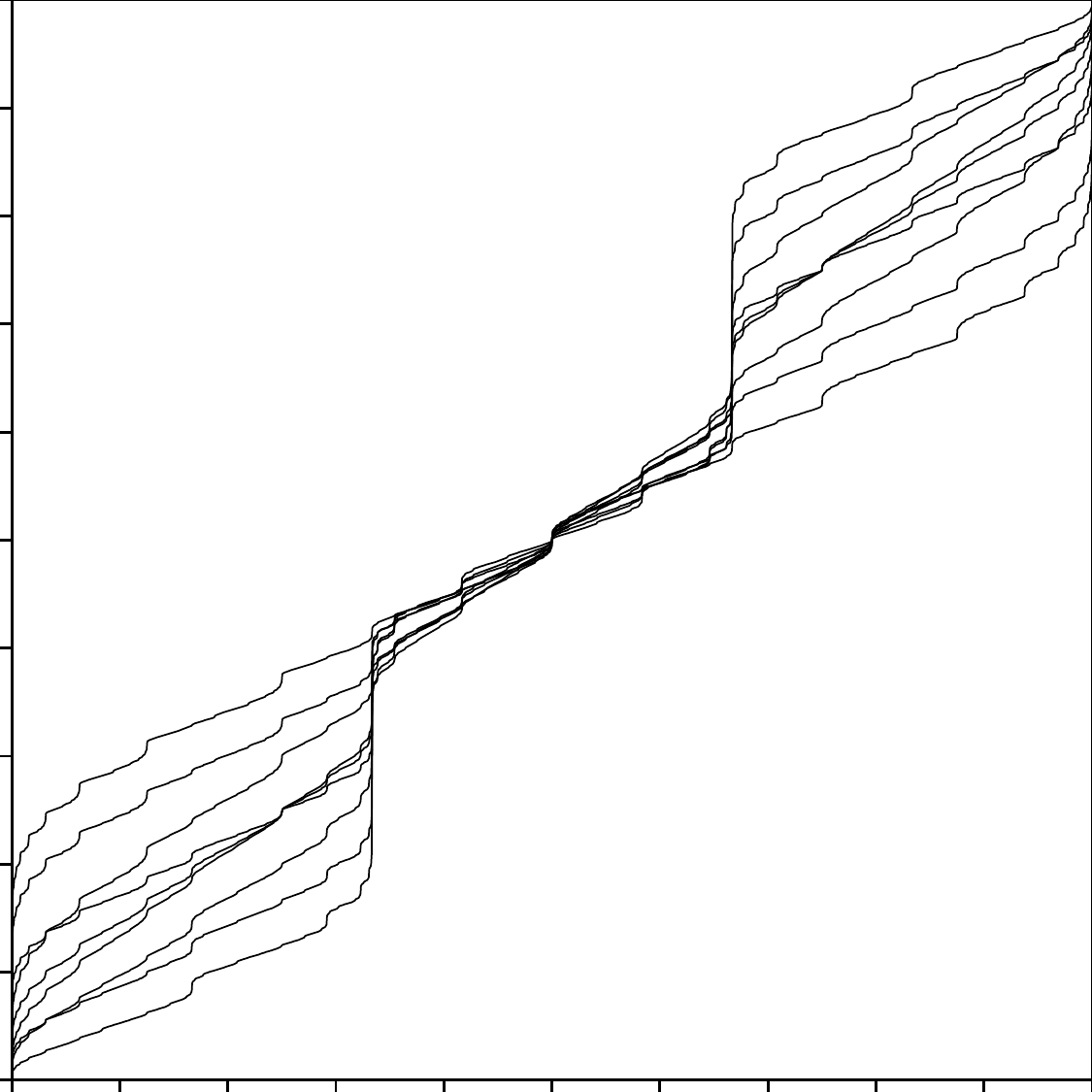}}
    \ffigbox[0.45\textwidth][][t]{}{\RawCaption{\caption{Plots of $F$ (restricted to $[0,1]$) when $\{X_n\}_{n\ge1}$ is either (a) a stationary Poisson process  (Example~\ref{ex:3}),   
    (b) an Ising model (Example~\ref{ex:symmetry}),
    (c) a renewal point processes (Example~\ref{ex:6}),
    (d) a mixed Poisson process (Example~\ref{ex:4}),
    or (e) a mixed 
    Markov chain (Example~\ref{ex:MCbeta}).
      The curves in (a)--(e) depend on the choice of parameter values as detailed in Examples~\ref{ex:3}, \ref{ex:symmetry}, \ref{ex:6}, \ref{ex:4}, and \ref{ex:MCbeta}. The axes run from 0 to 1.}\label{...}}}
    \end{subfloatrow*}
    }{}
\end{figure}

Our paper is organized as follows. 
Section~\ref{s:main results1} establishes in a general setting a useful expression of $F$ in terms of the finite dimensional distributions of $\{X_n\}_{n\geq 1}$. Further, assuming stationarity we provide sufficient and necessary conditions for the finite dimensional distributions in order to ensure either continuity of $F$ or that $F$ is strictly increasing at a point in $[0,1]$,
and we give useful results for the derivative of $F$. Then we turn to our main results for general model classes and study various examples: Section~\ref{s:MC-general} considers when 
$\{X_n\}_{n\geq 1}$ is a stationary Markov chain of some fixed order, where we show a law of pure type: either $F=F_1$ or $F$ is singular. Moreover, Section~\ref{s:MC-general} establishes sufficient conditions for continuity of $F$. These results together with those in Section~\ref{s:main results1} for monotonicity of $F$ and the expression of $F$ are exemplified for Bernoulli schemes (chains of order 0, i.e., when the $X_n$ are IID) and for binary Markov chains of order 1.
We also
 show that any singular continuous CDF can be approximated 
 by the CDF for a base-$q$ expansion with digits given by a stationary Markov chain of a sufficiently high order. 
 Section~\ref{s:renewal-general} considers when $\{X_n\}_{n\geq 1}$ is binary and $\{n\mid X_n=1\}$ constitutes 
 a stationary renewal process, where we establish similar results for $F$ as in Section~\ref{s:MC-general} and consider an example where $\{n\mid X_n=1\}$ is a determinantal point process. 
 Section~\ref{s:mixture-general} deals with mixture models, in particular  mixtures of stationary Markov chain of the same order and mixtures of stationary renewal processes, where we specify continuity and monotonicity conditions for $F$, consider its derivative $F'$, and study various concrete cases by expanding the examples of Sections~\ref{s:MC-general} and \ref{s:renewal-general}. Section~\ref{s:conclusion} summaries our findings in the previous sections and discusses some open problems. Finally, Appendix~A.1-A.11 contain the proofs of results in Sections~\ref{s:main results1}--\ref{s:mixture-general} together with some related technical results.

%

\section{Characterization of $F$ in terms of finite dimensional probabilities}\label{s:main results1}


We use the following terminology and notation. 
Define finite dimensional distributions of $\{X_n\}_{n\geq 1}$ by 
\begin{equation}\label{e:pdef}
    p(x_1,\ldots,x_n)\coloneqq\mathrm{P}(X_1=x_1,\dots,X_n=x_n),\quad n\in\mathbb N,\ x_1,\ldots,x_n\in\{0,\ldots,q-1\}.
\end{equation}
Denote the set of base-$q$ fractions in $(0,1)$ by $\mathbb Q_q$.
If $x=(0.x_1\ldots x_n00\ldots)_q\in\mathbb Q_q$ with $x_n>0$, we also have $x=(0.x_1\dots,x_{n-1}(x_n-1)(q-1)(q-1)\ldots)_q$ and we refer to $n$ as the \emph{order of $x$}, to $(0.x_1\dots x_n 00\ldots)_q$ as the \emph{terminating expansion of $x$}, and to $(0.x_1\ldots x_{n-1}(x_n-1)(q-1)(q-1)\dots)_q$ as the \emph{non-terminating expansion of $x$}.
If $q=2$ and $x=(0.x_1x_2\ldots)_2\in[0,1]$, define the corresponding \emph{point configuration} to $x$ by $y=\{n\in\mathbb N\,|\,x_n=1\}$ and the corresponding \emph{point process} to $\{X_n\}_{n\ge1}$ by 
\[Y\coloneqq\{n\in\mathbb N\,|\,X_n=1\}.\] 
Note that $\{X_n\}_{n\ge1}$ is determined by $Y$, since $X_n=1_Y(n)$. Furthermore, if $q=2$ and $n\in\mathbb N$, define $y_0\coloneqq\emptyset$, $Y_0\coloneqq\emptyset$, $y_n\coloneqq y\cap\{1,\ldots,n\}$, and
$Y_n\coloneqq Y\cap\{1,\ldots,n\}$. 

\begin{proposition}\label{t:useful}
For any $x=(0.x_1x_2\ldots)_q\in[0,1]$, where the non-terminating expansion is chosen if $x\in\mathbb Q_q$, we have
\begin{equation}\label{e:F-id1}
F(x)= 
\sum_{n\ge1:\,x_n\ge1}\sum_{k=0}^{x_n-1}p(x_1,\ldots,x_{n-1},k) + \rmP(X=x)
\end{equation}
(with $p(x_1,\ldots,x_{n-1},k)=p(k)$ if $n=1$).
If in addition $q=2$ and $y$ is the point configuration corresponding to $x$,  
then
\begin{equation}\label{e:F-id}
F(x)= 
\sum_{n\in y}\mathrm P(Y_{n-1}=y_{n-1},\, n\not\in Y) + \rmP(X=x).
\end{equation}
\end{proposition}

\begin{remark}\label{rem:useful}
	In contrast to other results in this paper, we do not assume stationarity in Proposition~\ref{t:useful}. If $x\in\mathbb Q_q$ and we consider its terminating expansion, then \eqref{e:F-id1}--\eqref{e:F-id} also hold provided we replace $\rmP(X=x)$ by $\rmP(X_1=x_1,\dots X_n=x_n,X_{n+1}=X_{n+2}=\dots=0)$. This follows by similar arguments as in Appendix~A.1. 
\end{remark}

\begin{proposition}\label{prop:useful:cont}
	Suppose that $\{X_n\}_{n\geq 1}$ is stationary and $x=(0.x_1x_2\ldots)_q\in [0,1]$.
	Then $F$ is continuous at $x$ if and only if $\lim_{n\rightarrow\infty}p(x_1,\ldots,x_n)=0$.
\end{proposition}

\begin{proposition}\label{prop:useful:derivative}
	Suppose that $\{X_n\}_{n\geq 1}$ is stationary. Then, for Lebesgue almost all $x\in[0,1]$, $F'(x)$ exists and equals a constant $c\in[0,1]$. Additionally, assume  
that $F$ is differentiable at $x=(0.x_1x_2\ldots)_q\in[0,1]\setminus\mathbb Q_q$. Then, for any $m\in \mathbb{N}\cup \{0\}$ and $\xi_1,\dots,\xi_m\in \{0,\dots,q-1\}$,
	\begin{equation}\label{e:F'}
		F'(x)=\lim_{n\rightarrow\infty}q^{n+m}p(x_1,\ldots,x_n,\xi_1,\dots,\xi_m),
	\end{equation}
	where for $m=0$ we replace $q^{n+m}p(x_1,\ldots,x_n,\xi_1,\dots,\xi_m)$ by $q^np(x_1,\dots,x_n)$. 
\end{proposition}

\begin{remark}\label{rem:fractions}
For Propositions~\ref{prop:useful:cont}--\ref{prop:useful:derivative} it should be noticed that $\mathbb Q_q$ is countable and stationarity implies that $F$ is continuous at any base-$q$ fraction (Theorem~1 in \cite{part1}) and so $X\not\in\mathbb Q_q$ almost surely. In statistical terms, 
\eqref{e:F'} with $m=0$ shows that $c$ may be interpreted as the limit of the likelihood ratio given by the probability of observing $x_1,\ldots,x_n$ under the stochastic model of $\{X_n\}_{n\ge1}$ 
versus under the IID case of $\{X_n\}_{n\ge1}$ with a uniform distribution on $\{0,\ldots,q-1\}$.
\end{remark}

\begin{corollary}\label{cor:useful:derivatives} 
Suppose that $\{X_n\}_{n\geq 1}$ is stationary and $F$ is differentiable at $x=(0.x_1x_2\ldots)_q\in(0,1)\setminus\mathbb Q_q$ such that $F'(x)>0$.
Then, for any $m\in \mathbb{N}\cup \{0\}$ and $\xi_1,\dots,\xi_m\in \{0,\dots,q-1\}$, 
\begin{equation}\label{e:F'2}
\lim_{n\to\infty} \frac{p(x_1,\dots,x_n,\xi_1,\dots,\xi_m)}{p(x_1,\dots,x_n,\xi_1,\dots,\xi_{m-1})}=q^{-1},
\end{equation}
where in the case $m=1$ we replace $p(x_1,\dots,x_n,\xi_1,\dots,\xi_{m-1})$ with $p(x_1,\dots,x_n)$ and in the case $m=0$ we replace $\frac{p(x_1,\dots,x_n,\xi_1,\dots,\xi_m)}{p(x_1,\dots,x_n,\xi_1,\dots,\xi_{m-1})}$ with $\frac{p(x_1,\dots,x_n)}{p(x_1,\dots,x_{n-1})}$.
\end{corollary}

\begin{remark}
Corollary~\ref{cor:useful:derivatives} becomes useful in Sections~\ref{s:MC-general} and \ref{s:renewal-general}
when considering Markov chains and renewal processes, since the left hand side in \eqref{e:F'2} then depends only on $\xi_1,\dots,\xi_m$. 
In statistical terms and considering $x_1,\ldots,x_n$ as data, \eqref{e:F'2} shows that for fixed $m\ge1$ and $n$ tending to infinity the predictive distribution of $X_{n+1},\ldots,X_{n+m}$ corresponds asymptotically to $m$ independent uniformly distributed random variables on $\{0,\ldots,q-1\}$.
\end{remark} 

The next proposition relates monotonicity properties of $F$ to the finite dimensional probabilities, where for any $x\in[0,1]$, we say that {\em $F$ is strictly increasing at $x$} if $F(s)<F(t)$ whenever $s<x<t$. 

\begin{proposition}\label{prop:useful:increasing}
	Suppose that $\{X_n\}_{n\geq1}$ is stationary and $x\in[0,1]$.
	Then $F$ is 
strictly increasing at $x$ if and only if for all $n\in \mathbb{N}$,
	\begin{itemize}
	\item if $x=(0.x_1x_2\ldots)_q\not\in\mathbb Q_q$, then
		\[ 
		p(x_1,\ldots,x_n) 
		>0,
		\] 
		\item   if $x=(0.x_1\dots x_m 0 0 \dots)_q\in\mathbb Q_q$ is of order $m$, then
		\[ 
		p(x_1,\ldots,x_m,\underbrace{0,\dots,0}_{n \textup{ zeros}})+p(x_1,\ldots,x_{m-1},x_m-1,\underbrace{q-1,\dots,q-1}_{n \textup{ } (q-1)\textup{'s}})>0.
		\] 
	\end{itemize}
In particular, $F$ is  
strictly increasing on $[0,1]$ if $p(x_1,\ldots,x_n)>0$ for all $n\in\mathbb N$ and all $x_1,\ldots,x_n\in\{0,\ldots,q-1\}$.
\end{proposition} 
In Proposition~\ref{prop:useful:increasing}, we can of course in all places replace `for all $n\in\mathbb N$' by `for all integers $n>n_0$, where $n_0>0$ is arbitrary'. 

\section{Markov chains}\label{s:MC-general}

For $m\in\mathbb N$, we say that $\{X_n\}_{n\geq 1}$ is a Markov chain of order $m-1$ if the $X_n$'s are independent when $m=1$, and $X_{n}$ conditioned on $X_1,\ldots,X_{n-1}$ depends only on $X_{n-m+1},\ldots,X_{n-1}$ whenever $n\geq m>1$. 
Assuming also stationarity, the finite dimensional distributions of $\{X_n\}_{n\ge1}$ are determined by initial probabilities  
$$\pi(x_1,\ldots,x_{m-1})=\mathrm P(X_1=x_1,...,X_{m-1}=x_{m-1})$$ if $m>1$ (and nothing if $m=1$) and transition probabilities 
$$\pi_{x_1,\ldots,x_{m}}=\mathrm P(X_m=x_m\,|\,X_1=x_1,...,X_{m-1}=x_{m-1})$$
(with $\mathrm P(X_m=x_m\,|\,X_1=x_1,...,X_{m-1}=x_{m-1})=\mathrm P(X_1=x_1)$ if $m=1$)  
satisfying the following obvious conditions (using the same notation $\pi$ for initial and transition probabilities becomes convenient in Section~\ref{s:mixture-general}, but notice that we use subscripts when considering transition probabilities). For all $x_1,\ldots,x_{m}\in\{0,\ldots,q-1\}$, we require  
\begin{itemize}
\item for $m=1$: $\pi_{x_1}\ge0$ and 
$\sum_{j=0}^{q-1}\pi_{j}=1$;
\item for $m>1$ and $\pi(x_1,\ldots,x_{m-1})\ge0$: $\sum_{j_1=0}^{q-1}\cdots\sum_{j_m=0}^{q-1}\pi(j_1,\ldots,j_{m-1})=1$, $\pi_{x_1,\ldots,x_{m}}\ge0$, and $\sum_{j=0}^{q-1}\pi_{x_1,\ldots,x_{m-1},j}=1$;
\item for $m>1$: the initial distribution should be invariant, that is,
\begin{equation}\label{e:equil}
\sum_{j=0}^{q-1}\pi(j,x_2,\ldots,x_{m-1})\pi_{j,x_2,\ldots,x_m}=\pi(x_2,\ldots,x_{m}),
\end{equation}
where $\pi(j,x_2,\ldots,x_{m-1})$ is replaced by $\pi(j)$ if $m=2$. 
\end{itemize}
 The conditions ensure that
the finite dimensional probabilities $p(x_1,\ldots,x_n)$ can be defined for $n\le m-1$ from $\pi(x_1,\ldots,x_{m-1})$, using its marginal distributions if $n<m-1$. Furthermore, for $n\ge m$, 
in accordance to the Markov property, we have
\begin{equation}\label{e:Markovfinitedimdist}
p(x_1,\ldots,x_n)=\pi(x_1,\ldots,x_{m-1})\prod_{j=m}^n \pi_{x_{j-m+1},\ldots,x_j},
\end{equation}
where $\pi(x_1,\ldots,x_{m-1})\coloneqq1$ if $m=1$. Finally, \eqref{e:equil} is needed because we assume stationarity.  

\begin{theorem}
\label{prop:markov_singular}
	Suppose that $\{X_n\}_{n\ge1}$ is a stationary Markov chain of an arbitrary order $m-1$. Then either $F$ is the uniform CDF on $[0,1]$ or
	$F$ is singular. Moreover, $F$ is continuous if $\pi_{x_1,\ldots,x_m}<1$ for all $x_1,\ldots,x_m\in\{0,\ldots,q-1\}$.
\end{theorem}

\begin{remark} 
Some comments to Theorem~\ref{prop:markov_singular} are in order. The theorem establishes a law of pure type, cf.\ the discussion at the beginning of Section~\ref{s:order 0}. If $\{X_n\}_{n\geq 1}$ is a stationary Markov chain $\{X_n\}_{n\geq 1}$ and $F$ is singular,
a natural question is if $F$ can be a mixture of $F_2$ and $F_3$ (in the notation of Section~\ref{s:setting}). This is indeed the case. For instance, take $q=3$, $m=2$,  $\pi(0)=\pi(1)=\pi(2)=1/3$, and 
	\[\begin{bmatrix*}[c]
		\pi_{0,0} & \pi_{0,1}& \pi_{0,2}\\
		\pi_{1,0} & \pi_{1,1}& \pi_{1,2}\\
		\pi_{2,0} & \pi_{2,1}& \pi_{2,2}
	\end{bmatrix*}=\begin{bmatrix*}[r]
		1/2 &0& 1/2 \\0& 1 & 0 \\ 1/2 & 0 & 1/2
	\end{bmatrix*}.\]
	Then $F=\tfrac13 F_2 + \tfrac23 F_3$ where $F_2(x)=H(x-\tfrac12)$ and $F_3$ is the Cantor function. Furthermore, a natural question is if
	the last condition in Theorem~\ref{prop:markov_singular} is a necessary condition. Indeed this is not the case. For example, let
	$q=m=2$, $\pi(0)=\tfrac23$, $\pi(1)=\tfrac13$, and  
		\[\begin{bmatrix*}[c]
			\pi_{0,0} & \pi_{0,1}\\
			\pi_{1,0} & \pi_{1,1}
		\end{bmatrix*}=\begin{bmatrix*}[r]
			1/2 & 1/2 \\ 1 & 0
		\end{bmatrix*}.\]
	Then $F$ is continuous, cf.\ Proposition~\ref{prop:useful:cont}. 
\end{remark}

The next proposition shows that any singular continuous CDF which satisfies the stationarity condition \eqref{e:3} can be approximated in the uniform norm by the CDF for the random variable on $[0,1]$ with digits given by a stationary Markov chain of sufficiently high order and defined in a natural way. 

\begin{proposition}\label{prop-horiac}
Assume $\{X_n\}_{n\ge1}$ is stationary and $F$ is singular continuous.
For $m=1,2,\ldots$, let $\{X_n^{(m)}\}_{n\ge1}$ be the stationary Markov chain of order $m-1$ which is obtained 
when $(X_1,\ldots,X_m)$ and $(X_1^{(m)},\ldots,X_m^{(m)})$ are identically distributed and a construction of the finite dimensional probabilities as in \eqref{e:Markovfinitedimdist} is used.  
Let $F^{(m)}$ be the CDF of $(0.X_1^{(m)}X_2^{(m)}\ldots)_q$. 
Then 
\begin{equation}\label{e:F3approx}
\lim_{m\rightarrow\infty}\sup_{x\in [0,1]}|F(x)-F^{(m)}(x)|= 0.
\end{equation}
\end{proposition}

Combining Proposition~\ref{prop-horiac} with the characterization results of stationarity from \cite{part1} (see Section~\ref{s:setting}) gives immediately the following result. 

\begin{corollary}
If $F$ satisfies \eqref{e:3}, then it can be approximated in the uniform norm by a countable mixture of CDF's coming from stationary Markov chains (possibly of different orders). 
\end{corollary}

\subsection{Bernoulli schemes}\label{s:order 0} 

Assume that $X_1,X_2,\ldots$ are IID (a stationary 0'th order Markov chain or a Bernoulli scheme) with distribution $(\pi_0,\ldots,\pi_{q-1})$. Seemingly this simple case is the most treated case in the literature on singular functions, and we discuss it
since various results easily follow from 
our general results above. Although it is a 
 well-studied case in probability theory, see e.g.\ Example 31.1 in \cite{Billingsley1995}, the authors in \cite{Okamoto2007} and \cite{Kyeonghee} do not acknowledge the simple probabilistic nature of the functions/CDFs they construct (namely that they are just considered the case of a Bernoulli scheme). 
 
By Theorem~\ref{prop:markov_singular}, $F$ is either the uniform CDF on $[0,1]$ (the case $\pi_0=\ldots=\pi_{q-1}=1/q$) or a singular CDF on 
$[0,1]$. This law of pure type is a well-known result due to Jessen and Wintner \cite{Jessen}, 
which states that any convergent infinite convolution of discrete measures is either singular or absolutely continuous with respect to Lebesgue measure. On the other hand, Theorem~\ref{prop:markov_singular} gives a law of pure type for the case of Markov chains of any order. 

Moreover, $F$ possesses the following properties.
For any $i\in\{0,\ldots,q-1\}$,  $F(x)=H(x-i/(q-1))$ if $\pi_i=1$. So suppose that $\pi_i<1$, $i=0,\ldots,q-1$.
Defining $0^0\coloneqq1$ and if $n_i(x_1,\ldots,x_n)$ is the number of times $x_1,\ldots,x_n$ are equal to $i$ (with $n_i(x_1,\ldots,x_{n-1},k)=n_i(k)$ if $n=1$), then for every $x=(0.x_1x_2\ldots)_q\in[0,1]$,
\[p(x_1,\ldots,x_n)=\prod_{i=0}^{q-1}\pi_i^{n_i(x_1,\ldots,x_n)}\]
and
\begin{equation}\label{e:FIID}
F(x)= 
\sum_{n\ge1:\,x_n\ge1}\sum_{k=0}^{x_n-1}\prod_{i=0}^{q-1}\pi_i^{n_i(x_1,\ldots,x_{n-1},k)},
\end{equation}
 cf.\ Proposition~\ref{t:useful}. Thus $F$ is continuous, cf. Theorem~\ref{prop:markov_singular},  
 and 
 $F$ is strictly increasing at $x\in[0,1]$ if and only if 
$\pi_{x_j}>0$ for every digit $x_j$ which appears in a base-$q$ expansion of $x$ (i.e., one of the two base-$q$ expansions of $x$ if $x\in\mathbb Q_q$), cf.\ Proposition~\ref{prop:useful:increasing}.


\begin{example}[Riesz-Nagy functions]\label{ex:3} 
	Consider the dyadic case $q=2$. Then  
	we call $Y$ (as defined at the beginning of Section~\ref{s:main results1}) a stationary Poisson process (conditioned on having no multiple points). 
	Writing $\pi_{1}=\pi$, let us assume $0<\pi<1$ and set $\alpha\coloneqq\pi/(1-\pi)$. 
	For any $x=(0.x_1x_2\ldots)_2\in[0,1]$ and corresponding point configuration $y$,
	\begin{equation}\label{e:F-special1}
	F(x)=\sum_{n\in y}\pi^{\# y_{n-1}}(1-\pi)^{n-\# y_{n-1}}=\sum_{n\in y}\frac{\alpha^{\# y_{n-1}}}{(1+\alpha)^n},
	\end{equation}
	cf.\ \eqref{e:F-id}.
	This CDF is called a Riesz-Nagy function, though it was already introduced in \cite{Cesaro1906} (without use of functional equations or probabilistic constructions).
The expression \eqref{e:F-special1} was also established in 
\cite{Rham1956} but as the solution to a certain functional equation for a bounded function, and in		
		\cite{Takacs1978} as the explicit form for the geometric construction in \cite{Salem1943} and \cite{Riesz1955} (see also 
\cite{Kairies1997} and	 
the neat probabilistic exposition in Example 31.1 in \cite{Billingsley1995}).
Note that $F$ is strictly increasing on $[0,1]$, it is the uniform CDF on $[0,1]$ if $\pi=\frac12$, and it is singular continuous otherwise. Denoting $F$ in \eqref{e:F-special1} by $F_\pi$, we have 
\begin{equation}\label{e:qqqqqqqq}
F_\pi(x)+F_{1-\pi}(1-x)=1,
\end{equation}
since $1-X=\sum_{k=1}^\infty (1-X_k)2^{-k}$ follows $F_{1-\pi}$.
Figure~1(a) shows plots of $F_\pi$ for $\pi=0.1,0.2,\dots$, 0.9 (from the top to the bottom).
\end{example}


\begin{example}[Cantor function and related cases]\label{ex:cantor}	
	Consider the triadic case $q=3$ and   
	suppose that  $0<\pi_0<1$ and $\pi_1=0$. Then the distribution of $X$ is concentrated on the Cantor set 
	\[C=\{(0.x_1x_2\ldots)_3\,|\,x_1,x_2,\ldots\in\{0,2\}\}=[0,1]\setminus\bigcup_{n=1}^\infty\bigcup_{k=1}^{2^{n-1}}I_{k,n}, \]
	where each $I_{k,n}$ is one of the $2^{n-1}$ intervals of the form $(a_k+ 1/3^n, a_k + 2/3^n)$, where $a_k = \sum_{i=1}^{n-1}x_i 3^{-i}$ and $x_i$ is either $0$ or $2$. For any $x=(0.x_1x_2\ldots)_3\in[0,1]$, setting $\gamma\coloneqq\pi_0/(1-\pi_0)$ and $n_x\coloneqq\inf\{n\in\mathbb N\,|\,x_n=1\}$ with $\inf\emptyset\coloneqq\infty$,
	we obtain from \eqref{e:F-id1} and a straightforward calculation
	that
	\begin{align*}
	F(x)=\sum_{1\le n< n_x:\,x_n=2}\frac{\gamma^{n_0(x_1,\dots,x_{n-1})+1}}{(1+\gamma)^{n}}
	+
	\begin{cases}
	\frac{\gamma^{1+n_0(x_1,\ldots,x_{n_x-1})}}{(1+\gamma)^{n_x}}    
	& \mbox{if }n_x<\infty,\\
	0 & \mbox{else}.
	\end{cases}
	\end{align*}
	Here $F$ is singular continuous, strictly increasing on $C$, and constant on each connected component of $[0,1]\setminus C$ (the union of the removed middle thirds 
	$I_{k,n}$). 
	Particularly, if $\pi_0=\frac12$, then 
	\[ 
	F(x)=2^{-n_x}+
	\sum_{1\le n< n_x:\,x_n=2}2^{-n}
	\] 
	is the Cantor function, cf.\ Equation (1.2) and Corollary~5.9 in \cite{DOVGOSHEY2006}. %
\end{example}

\subsection{Binary Markov chains of order 1}\label{s:binaryMC}

	Assume that $q=2$ and $\{X_n\}_{n\ge1}$ is a stationary ergodic 
	Markov chain of order 1, that is,
	\begin{equation}\label{e:two-state-MC}
	\pi_{0}=\frac{p_0}{p_0+p_1},\qquad \pi_{1}=\frac{p_1}{p_0+p_1},\qquad
	\begin{bmatrix}
	\pi_{0,0} & \pi_{0,1} \\
	\pi_{1,0} & \pi_{1,1}
	\end{bmatrix}=
	\begin{bmatrix}
	1-p_1 & p_1 \\
	p_0 & 1-p_0
	\end{bmatrix}, 
	\end{equation}
	where $0<p_0\le1$, $0<p_1\le1$, and at least one of $p_0$ and $p_1$ is strictly less than $1$.  
By \eqref{e:F-id}, 
	for any $x\in[0,1]$, 
	we have  
	\begin{align}
	F(x)=\pi_0^{1-x_1}\pi_1^{x_1}\sum_{n:\, x_n=1}p_0^{x_{n-1}}(1-p_1)^{1-x_{n-1}}\prod_{i,j=0}^1\pi_{i,j}^{n_{i,j}(x_1,\ldots,x_{n-1})}
	\label{e:wwww}
	\end{align}
	where $0^0\coloneqq1$ and $n_{i,j}(x_1,\ldots,x_{n-1})$ is the number of times the sequence $i,j$ occurs in the sequence $x_1,\ldots,x_{n-1}$ (if $n\le2$ we set $n_{i,j}(x_1,\ldots,x_{n-1})$ equal to 0). 
	Here
	$F$ is either the uniform CDF on $[0,1]$ (the case $p_0=p_1=\tfrac12$) or singular continuous, cf.\  Theorem~\ref{prop:markov_singular} (this also follows from \cite{Dym1968} since $\{X_n\}_{n\ge1}$ is stationary and ergodic).	
	 Furthermore, $F$ is not strictly increasing at $x=(0.x_1x_2\ldots)_2\in[0,1]$ if and only if either  $p_0=1$ and $x_k=x_{k+1}=1$ for some $k\ge1$, or $p_1=1$ and $x_k=x_{k+1}=0$ for some $k\ge1$. This follows from \eqref{e:two-state-MC} 
	 and 
	 Proposition~\ref{prop:useful:increasing}. 

\begin{example}[Ising model]\label{ex:symmetry}
	Let
	$p_0=p_1=\pi\in(0,1)$, so 
	$\pi_0=\pi_1=\tfrac12$, $\pi_{0,0}=\pi_{1,1}=1-\pi$, and $\pi_{0,1}=\pi_{1,0}=\pi$. Defining $\beta\coloneqq\pi/(1-\pi)$, then 
	\[
	p( x_1,\ldots,x_n)=
	\frac12(1+\beta)^{1-n} \beta^{\sum_{i=1}^{n-1}1[x_i\not=x_{i+1}]},
	\]
which is the probability mass function of a finite Ising model 
\cite{Ising1925}.
For every $x\in[0,1]$, \eqref{e:wwww} reduces to
	\begin{equation}\label{e:ppp}
	F(x) 
	=
	\frac12
	\sum_{n:\, x_n=1}\frac{\beta^{m(x_1,\ldots,x_{n-1})}}{(1+\beta)^{n-1}},
	\end{equation}
	where 
	$m(x_1,\ldots,x_{n-1})$ is the number of switches in the sequence $x_1,\ldots,x_{n-1}$ (if $n\le2$ we set $m(x_1,\ldots,x_{n-1})$ equal to 0). 
	Here $F$ is strictly increasing on $[0,1]$, it is the uniform CDF on $[0,1]$ if $\pi=\frac12$, and it is singular continuous otherwise. As in Example~\ref{ex:3} we see that \eqref{e:qqqqqqqq} is satisfied when $F_\pi$ now denotes  $F$ in \eqref{e:ppp}.
	Figure~1(b) shows plots of $F_\pi$ for $\pi=0.1,0.2,\ldots,0.9$ (from the top to the bottom when considering the left part of the curves and from the bottom to the top when considering the right part of the curves).	
\end{example}


\section{Renewal processes}\label{s:renewal-general}

Let $q=2$ and $Z_0,Z_1,Z_2,\ldots$ be independent random variables with state space $\mathbb N$, where $Z_1,Z_2,\ldots$ are identically distributed with finite mean $\mu$, and $\mathrm P(Z_0=n)=\mathrm P(Z_1\ge n)/\mu$ for $n\in\mathbb N$. In other words, $Y=\{Z_0,Z_0+Z_1,Z_0+Z_1+Z_2,\ldots\}$ is a stationary delayed renewal process, see 
Section 2.4 in \cite{Soshnikov2000}, Example 1.7 in \cite{LyonsSteif2003}, 
and the references therein. Considering $\{X_n\}_{n\ge1}$, with $X_n=1_Y(n)$, this process is stationary. 
The following theorem characterizes $F$ in terms of the distribution of $Z_1$.

\begin{theorem}\label{prop:renewal_singular}
	Let 
	$\{X_n\}_{n\geq 1}$ be the binary process obtained from a stationary delayed renewal process $Y$ as described above. There are three cases:
\begin{enumerate} 
\item[{\rm (I)}] If $Z_1$ is geometrically distributed with mean $2$, then 
$F$ is the uniform CDF on $[0,1]$.
\item[{\rm (II)}] If $Z_1$ is degenerated, i.e., $\rmP(Z_1=k)=1$ for some $k\in \mathbb{N}$, then $F$ is the 
uniform distribution on the set $\{2^{-\ell}/(1-2^{-k})\mid\ell\in\{1,\ldots,k\}\}$. 
\item[{\rm (III)}] Otherwise $F$ is singular continuous. 
\end{enumerate}
Moreover,
\begin{enumerate}
\item[{\rm (IV)}] $F$ is strictly increasing on $[0,1]$ if and only if $\mathrm P(Z_1=n)>0$ for all $n\in\mathbb N$.
\end{enumerate}
\end{theorem}

\begin{remark} Theorem~\ref{prop:renewal_singular} is also establishing that $F$ is of pure type. The case (I) corresponds of course to the `fair coin case' (the $X_n$ are IID with $\mathrm P(X_n=0)=\mathrm P(X_n=1)=\tfrac12$). 
\end{remark} 

 In order to describe the finite dimensional probabilities we need the following notation.
 For any $x=(0.x_1x_2\ldots)_2\in[0,1]$ and corresponding point configuration $y$ (see the lines below \eqref{e:pdef}), if $z_0+\ldots+z_k$ is the $(k+1)$'th point in $y$, we define $\sup\emptyset\coloneqq-1$ and
\begin{align*}
m_n&=m(x_1,\ldots,x_n)\coloneqq\sup\{k\in\{0,1,\ldots\}\,|\,z_0+\ldots+z_k\le n\}\\
&=
\begin{cases} -1&\mbox{if }x_1=\ldots=x_n=0\\
x_1+\ldots+x_n &\mbox{otherwise}\end{cases}
\end{align*} 
meaning that in the latter case $m_n$ is the number of points before time $n$.
Then
\begin{align}\label{eq:renewal_probability}
p(x_1,\ldots,x_n)=\mathrm P(Z_{m_n+1}>n-z_0-\ldots-z_{m_n})\prod_{\ell=0}^{m_n}\mathrm P(Z_\ell=z_\ell),
\end{align}
which in the case $m_n=-1$ 
is interpreted as $p(x_1,\ldots,x_n)=\mathrm P(Z_0>n)$. 


The following example of a renewal process is in fact also an example of a determinantal point process, see Example 1.7 in \cite{LyonsSteif2003} (see also \cite{Soshnikov2000}). 

\begin{example}[A special renewal process]\label{ex:6}
	Suppose $Z_1-1$ is negative binomially distributed with parameters 2 and $\pi\in(0,1)$, so 
	\[\mathrm P(Z_1=n)=n(1-\pi)^2\pi^{n-1},\qquad n\in\mathbb N,\]
	and $\mu=(1+\pi)/(1-\pi)<\infty$.
Then 
	\begin{align*}
	\mathrm P\left(Z_1>z_1\right) 
	&=(1-\pi)(z_1+1)\pi^{z_1}+\pi^{z_1+1},\\
	\mathrm P(Z_0=z_0)&=\frac{(1-\pi)^2}{1+\pi}z_0\pi^{z_0-1} + \frac{1-\pi}{1+\pi}\pi^{z_0},\\
	\mathrm P\left(Z_0>z_0\right)
	&=\frac{1-\pi}{1+\pi}(z_0+1)\pi^{z_0}+2\frac{\pi^{z_0+1}}{1+\pi},
	\end{align*}
	where both $Z_0$ and $Z_1$ have support equal to $\mathbb N$. 
	Thus, for any $x\in(0,1]$,
	   we obtain from \eqref{e:F-id} and \eqref{eq:renewal_probability}  that
	\begin{equation}\label{eq:renewal_closed_form}
		\begin{aligned}
	F(x)=&\,\frac{1-\pi}{1+\pi}(z_0+1)\pi^{z_0}+\frac{2}{1+\pi}\pi^{z_0+1}\\
	&\,+\left[\frac{(1-\pi)^2}{1+\pi}z_0\pi^{z_0-1}+\frac{1-\pi}{1+\pi}\pi^{z_0}\right]\\
&\,	\times\sum_{k=1}^\infty \left[(1-\pi)(z_k+1)+\pi\right](1-\pi)^{2k-2}\pi^{z_1+\ldots z_k-k+1}\prod_{\ell=1}^{k-1}z_\ell.
	\end{aligned}
	\end{equation}	
	Here  $F$ is singular continuous and strictly increasing on $[0,1]$, cf.\ Theorem~\ref{prop:renewal_singular}.  
Denoting $F$ in \eqref{eq:renewal_closed_form} by $F_\pi$, Figure~1(c) shows plots of $F_\pi$ for $\pi=0.1,0.2,\ldots,0.9$ (from the bottom to the top).	
\end{example}

\section{Mixtures of stationary processes}\label{s:mixture-general} 

In this section we consider $F$ to be a mixture of CDFs corresponding to random variables with stationary digits. Therefore, 
 we imagine that $\Pi$ is some random variable (or `random parameter') used to specify the distribution of $\{X_n\}_{n\ge1}$ conditioned on $\Pi$, and write
\[F_\Pi(x)\coloneqq \mathrm P(X\le x\mid\Pi),\qquad x\in\mathbb R,\] 
so that the (unconditional) CDF of $X$ is given by 
\begin{equation}\label{e:mixtureF}
F=\mathrm EF_\Pi.
\end{equation}
Examples of such mixture models will be given in Sections~\ref{s:mixture-MC-general}--\ref{s:mixture-renewal-general}, where the two simplest examples are Examples~\ref{ex:4} and \ref{ex:MCbeta} below: Briefly, they relate to the Poisson model case in Example~\ref{ex:3} and the Ising model case in Example~\ref{ex:symmetry}, respectively, by replacing the probability parameter $\pi$ by a random variable following a beta distribution so that we obtain a mixed Poisson process in Example~\ref{ex:4} and a mixed Markov chain in Example~\ref{ex:MCbeta}. Corresponding plots of $F$ for various choices of beta distributions are seen in  Figure~1(d)-(e) and they show a different behaviour as compared to the Poisson and Ising model cases shown in
 Figure~1(a)-(b) (as well as to Figure~1(c)). We also find the present section interesting for more theoretical reasons: Indeed  our main results (Theorems~\ref{prop:b1} and \ref{prop:b111}) contain the deepest results of this paper, cf.\ Appendix~A.10-A.11. 

  
We need some notation. We denote the state space of $\Pi$ by $\Omega$ and assume it is equipped with a $\sigma$-algebra $\mathcal A$ depending on the context; specific examples are given in Section~\ref{s:mixture-MC-general} 
and  
in Appendix~A.10-A.11 
we specify $\Omega$ in the general context of a Markov chain and a renewal process as considered in Propositions~\ref{prop:b1}--\ref{prop:b111}. For $\pi\in\Pi$ and $x\in\mathbb R$, we let $F_\pi(x)\coloneqq \mathrm P(X\le x\mid\Pi=\pi)$.
For $x_1,\ldots,x_n\in\{0,\ldots,q-1\}$, we write 
 \[p_\Pi(x_1,\ldots,x_n)\coloneqq \mathrm P(X_1=x_1,\ldots,X_n=x_n\mid \Pi)\]
so that  
\begin{equation}\label{e:mixture-p}
p(x_1,\ldots,x_n)=\mathrm Ep_\Pi(x_1,\ldots,x_n)
\end{equation} 
specifies the (unconditional) finite dimensional probabilities of $\{X_n\}_{n\ge1}$. Similarly, for $\pi\in\Pi$, we let $p_\pi(x_1,\ldots,x_n)\coloneqq \mathrm P(X_1=x_1,\ldots,X_n=x_n\mid \Pi=\pi)$.
 For any $x\in[0,1]$, define  
$$A_x\coloneqq\{ \pi\in\Omega\mid F_\pi \textup{ is discontinuous at } x\}$$
and
 $$B_x\coloneqq\{\pi\in\Omega\mid F_\pi\mbox{ is strictly increasing at }x\}.$$
We assume that $A_x\in\mathcal A$, $B_x\in\mathcal A$, and $\{\pi\in\Omega\mid F_\pi=F_1\}\in\mathcal A$ (this will be obviously satisfied for the specific examples considered in Section~\ref{s:mixture-MC-general}). Finally, let $1[\cdot]$ denote the indicator function. 

The following proposition is a straightforward consequence of Propositions~\ref{t:useful}--\ref{prop:useful:increasing}, the law of total expectation, and Lebesgue's dominated convergence theorem, where we recall that $F_1$ denotes the uniform CDF on $[0,1]$ (cf.\ Section~\ref{s:setting}) and when $q=2$, $Y$ is the point process corresponding to $X$ (cf.\ the beginning of Section~\ref{s:main results1}). 
 
\begin{proposition}\label{prop:mixture-general} 
 Assume $F_\Pi$ satisfies the stationarity condition \eqref{e:3} almost surely. Then
 $F$ given by \eqref{e:mixtureF} has the following properties.  
\begin{enumerate}
\item[{\rm (I)}] $F$ is a CDF on $[0,1]$ which satisfies \eqref{e:3}.
\item[{\rm (II)}] For any $x=(0.x_1x_2\ldots)_q\in[0,1]$, 
\[ 
F(x)= 
\sum_{n\ge1:\,x_n\ge1}\sum_{k=0}^{x_n-1}\mathrm Ep_\Pi(x_1,\ldots,x_{n-1},k) 
\] 
(with $Ep_\Pi(x_1,\ldots,x_{n-1},k)=Ep_\Pi(k)$ if $n=1$).
If in addition $q=2$ and $y$ is the point configuration corresponding to $x$,  
then
\[ 
F(x)= 
\sum_{n\in y}\mathrm E\mathrm P(Y_{n-1}=y_{n-1},\, n\not\in Y\mid\Pi).
\] 
\item[{\rm (III)}] For any $x\in[0,1]$, $F$ is continuous at $x$ if and only if $\mathrm P(\Pi\in A_x)=0$.
\item[{\rm (IV)}] For any $x\in[0,1]$, $F$ is strictly increasing at $x$ if and only if $\mathrm P(\Pi\in B_x)>0$.
\item[{\rm (V)}]   For any $x\in \mathbb{R}$, define
\begin{equation}\label{eq:G}
	G(x)\coloneqq\mathrm{E}( F_\Pi(x) 1[F_\Pi\neq F_1]).
\end{equation} 
Then $F$ is differentiable at $x$ if and only if $G$ is differentiable at $x$. Furthermore, for  Lebesgue almost all $x=(0.x_1x_2\ldots)_q\in[0,1]$, 
\begin{align}
F'(x)&= \rmP(F_\Pi=F_1)+G'(x) 
\label{e:min}
\end{align}
is equal to a constant in $[0,1]$. 
\end{enumerate}
\end{proposition} 

In connection to \eqref{e:min} a natural question is when $G'(x)=0$: For every $n\in\mathbb N$ and every $x=(0.x_1x_2\ldots)_q\in[0,1]\setminus\mathbb Q_q$, consider the function $f_{x,n}\colon \Omega\mapsto[0,q^n]$ defined by
\[ 
	f_{x,n}(\pi)=q^np_\pi(x_1,\dots,x_n)1[F_\pi\neq F_1].
\] 
If $G\not=0$, then $G$ is proportional to a CDF satisfying \eqref{e:3}, and hence Proposition~\ref{prop:useful:derivative} and \eqref{e:mixture-p} imply that for Lebesgue almost all $x\in[0,1]\setminus\mathbb Q_q$,
\begin{equation}\label{e:tosset}
	G'(x) 
	=\lim_{n\to\infty} \mathrm{E}( f_{x,n}(\Pi)).
\end{equation}
However, it is not obvious if we can interchange the limit and the expectation in \eqref{e:tosset} and hence obtain a limit which is 0. The following proposition provides conditions ensuring that $G'=0$ almost everywhere on $[0,1]$.

\begin{proposition}\label{cor:G1G2}
 Let $F$ be given by \eqref{e:mixtureF} and suppose $F_\Pi$ satisfies the stationarity condition \eqref{e:3} almost surely. Assume there exist sets $V_1\subseteq V_2\subseteq\dots$ in $\mathcal{A}$ such that $\bigcup_{j=1}^\infty V_j=\{\pi\in \Omega \mid F_\pi\neq F_1\}$ and for all $j\in \mathbb{N}$ and Lebesgue almost all $x\in[0,1]\setminus\mathbb Q_q$, the sequence $\{f_{x,n}(\cdot)1[\cdot\in V_{j}]\}_{n\geq 1}$ is pointwise convergent to $0$ as $n\to\infty$ and dominated by a constant which may depend on $(x,j)$. Then $F'=\rmP(F_\Pi=F_1)$ almost everywhere on $[0,1]$. 
\end{proposition}

\begin{remark} We apply Proposition~\ref{cor:G1G2} in the proofs of the theorems in Sections~\ref{s:mixture-MC-general}--\ref{s:mixture-renewal-general}. As in Proposition~\ref{cor:G1G2} these theorems give that $F$ is of pure type if and only if either $F_\Pi=F_1$ almost surely or $F_\Pi\not=F_1$ almost surely.
\end{remark}   

\subsection{Mixtures of Markov chains}\label{s:mixture-MC-general} 

The following proposition considers the mixture of CDFs corresponding to stationary Markov chains of a fixed order.

\begin{theorem}\label{prop:b1}
	Let $m\in\mathbb N$ and suppose $\{X_n\}_{n\ge1}$ conditioned on $\Pi$ is a stationary Markov chain of order $m-1$.
	Then we have the following for $F$ given by the mixture CDF in \eqref{e:mixtureF}.
	\begin{enumerate} 
		\item[{\rm (I)}] 
		For any $x\in [0,1]$, $F$ has a discontinuity at $x$ if and only if $\mathrm P(\Pi\in A_x)>0$. 
		\item[{\rm (II)}] For any $x\in [0,1]$, $F$ is strictly increasing at $x$ if and only if $\mathrm P(\Pi\in B_x)>0$.
		\item[{\rm (III)}]  
		 For  Lebesgue almost all $x\in[0,1]$, $F'(x)=\mathrm P(F_\Pi=F_1)$.   
	\end{enumerate}
\end{theorem}


\begin{example}[Mixtures of Bernoulli scheme and Poisson process constructions]\label{ex:4}
	Let $\Pi$  
	be a $q$-dimensional random vector with state space $\Omega=\Delta_q$, the $(q-1)$-dimensional simplex consisting of all probability distributions $\pi=(\pi_0,\ldots,\pi_{q-1})$. Since $\pi_0=1-\sum_{i=1}^{q-1}\pi_i$, we can  
 identify $\Delta_q$ by the set $B(q)\coloneqq\{(\pi_1,\ldots,\pi_{q-1})\mid\sum_{i=1}^{q-1}\pi_i\le1\}$, which we equip with the Borel $\sigma$-algebra. Suppose that $X_1,X_2,\ldots$ conditioned on $\Pi$ are
	IID with distribution $\Pi$.

	Combining \eqref{e:mixtureF} and  \eqref{e:FIID} gives for every $x=(0.x_1x_2\ldots)_q\in[0,1]$,
	\begin{equation}\label{e:FmixDirichlet}
	F(x)= 
	\sum_{n\ge1:\,x_n\ge1}\sum_{k=0}^{x_n-1}\mathrm E\prod_{i=0}^{q-1}\Pi_i^{n_i(x_1,\ldots,x_{n-1},k)}.
	\end{equation}
	We have $F'=\mathrm P(\Pi=(1/q,\ldots,1/q))$  almost everywhere on $[0,1]$, cf.\ Theorem~\ref{prop:b1}(III).
	Furthermore, $F$ is strictly increasing at $x\in [0,1]$ if and only if all $\Pi_{x_j}>0$ almost surely for every digit $x_j$ which occurs in a base-$q$ expansion of $x$, cf.\ Section~\ref{s:order 0} and Proposition~\ref{prop:mixture-general}(IV). 
	
	It follows from Section~\ref{s:order 0} that for any $x\in[0,1]$, $F_\pi$ is discontinuous at $x$ if and only if for some $k\in\{0,\ldots,q-1\}$ we have that $x=k/(q-1)$ and $\pi=\pi^{(k)}$, where $\pi^{(k)}=(\pi^{(k)}_0,\ldots,\pi^{(k)}_{q-1})$ is the $k$th standard vector (i.e., the vertex of $\Delta_q$ given by $\pi^{(k)}_k=1$). Note that $A_x=\{\pi^{(k)}\}$ consists of a singleton if $x=k/(q-1)$ with $k\in\{0,\ldots,q-1\}$, and  $A_x$ is empty otherwise. Hence, by Theorem~\ref{prop:b1}(I), $F$ has a discontinuity at $x\in[0,1]$ if and only if $x=k/(q-1)$ where $k\in\{0,\ldots,q-1\}$ and $\mathrm P(\Pi=\pi^{(k)})>0$. In particular,  $F$ is singular continuous if $\Pi$ is a continuous random variable.
 
 For example, suppose that $\Pi$ follows a 
 Dirichlet distribution with shape parameters $\beta_0>0,\ldots,\beta_{q-1}>0$ and probability density function given by
	\begin{equation}\label{e:dirichlet}
\frac{\Gamma(\sum_{i=0}^{q-1}\beta_i)}{\prod_{i=0}^{q-1}\Gamma(\beta_i)}\prod_{i=0}^{q-1}\pi_i^{\beta_i-1},\qquad (\pi_1,\ldots,\pi_{q-1})\in B(q).	
\end{equation}	
 It follows from the considerations above that
  $F$ is singular continuous and strictly increasing on $[0,1]$. Moreover, we can easily evaluate the expected value in \eqref{e:FmixDirichlet}. 
  
	For instance, if $q=2$,
	then $Y$ is 
	a mixture of stationary Poisson processes (conditioned on having no multiple points). If furthermore $\Pi_0$ follows a beta distribution (that is, $(\Pi_0,\Pi_1)$ follows a Dirichlet distribution
	with shape parameters $\beta_0>0$ and $\beta_1>0$), then for any $x\in[0,1]$ it follows from \eqref{e:FmixDirichlet} and \eqref{e:dirichlet} that  
	\begin{equation}\label{e:F-beta}
	F(x)=\sum_{n\in y}\frac{B(\beta_0+n-\#y_{n-1},\beta_1+\#y_{n-1})}{B(\beta_0,\beta_1)},
	\end{equation}
	where $B(\beta_0,\beta_1)=\Gamma(\beta_0)\Gamma(\beta_1)/\Gamma(\beta_0+\beta_1)$ is the beta function.
	This provides a large parametric family of CDFs, where
	\eqref{e:F-beta} agrees with \eqref{e:F-special1} in the limit as $\beta_0\rightarrow\infty$ and $\beta_1\rightarrow\infty$ such that $\beta_1/\beta_0\rightarrow \alpha$. It follows immediately from above that 
	$F$ in \eqref{e:F-beta} is singular continuous and strictly increasing on $[0,1]$. Denoting this CDF by $F_{\beta_0,\beta_1}$, we obtain from \eqref{e:qqqqqqqq} and \eqref{e:mixtureF} that
	\begin{equation}\label{e:Fbeta1beta2etc}
	F_{\beta_0,\beta_1}(x)+F_{\beta_1,\beta_0}(1-x)=1.
	\end{equation}
Figure~1(d) shows plots of $F_{\beta_0,\beta_1}$ when $\beta_0,\beta_1\in \{0.5,1,1.5\}$, where the curves overlap and hence are not so easy to distinguish but they are determined by that 
	\begin{align*}F_{0.5,1.5}(0.25)&>F_{0.5,1.0}(0.25)>F_{1.0,1.5}(0.25)>F_{0.5,0.5}(0.25)> F_{1.0,1.0}(0.25)\\
	&> F_{1.5,1.5}(0.25)>F_{1.5,1.0}(0.25)>F_{1.0,0.5}(0.25)>F_{1.5,0.5}(0.25).\end{align*}
\end{example}


\begin{example}[Mixture of Ising models]\label{ex:MCbeta}
	Along similar lines as in Example~\ref{ex:4}, we can impose a distribution on the parameters of a Markov chain of order 1 and obtain results in a similar way. For example, 
	suppose that $\pi$ in Example~\ref{ex:symmetry} is replaced by a random variable $\Pi$ which is beta distributed  with shape parameters $\beta_0>0$ and $\beta_1>0$. For any $x\in[0,1]$, combining \eqref{e:mixtureF} and \eqref{e:ppp} gives 
\begin{align}\label{eq:beta_markov}
	F(x)=\frac12
	\sum_{n\in y}\frac{B(\beta_0+m(y_{n-1}),\beta_1+n-1-m(y_{n-1}))}{B(\beta_0,\beta_1)}.
	\end{align}	
	This CDF is singular continuous and strictly increasing  on $[0,1]$: By Theorem~\ref{prop:b1}(I), $F$ is continuous at any $x\in[0,1]$, since $A_x$ is empty (see Section~\ref{s:binaryMC}); it follows immediately from Example~\ref{ex:symmetry}
	and Theorem~\ref{prop:b1}(II) that $F$ is strictly increasing  on $[0,1]$; and it follows from Theorem~\ref{prop:b1}(III) that $F'=0$ almost everywhere.
	Denoting $F$ in \eqref{eq:beta_markov} by $F_{\beta_0,\beta_1}$, it follows from Example~\ref{ex:symmetry}  and \eqref{e:mixtureF} that $F_{\beta_0,\beta_1}$ satisfies \eqref{e:Fbeta1beta2etc}.
	Figure~1(e) shows plots of $F_{\beta_0,\beta_1}$ when $\beta_0,\beta_1\in \{0.5,1,1.5\}$, where again the curves are not easy to distinguish but they are determined by that
	\begin{align*}F_{0.5, 1.5}(1/8)&>F_{0.5,1.0}(1/8)>F_{1.0,1.5}(1/8)>F_{0.5,0.5}(1/8)> F_{1.0,1.0}(1/8)\\
	&> F_{1.5,1.5}(1/8)>F_{1.5,0.5}(1/8)>F_{1.0,0.5}(1/8)>F_{1.0,1.5}(1/8).\end{align*}
\end{example}

\subsection{Mixtures of renewal processes}\label{s:mixture-renewal-general} 

 In this section we assume that $\{X_n\}_{n\ge1}$ conditioned on $\Pi$ is a binary process obtained from a stationary delayed renewal process as in Section~\ref{s:renewal-general}. We can think of value of $\Pi$ as specifying the 
 distribution of $Z_1$, where the only requirement is that $\mathrm EZ_1<\infty$. For each $k\in\mathbb N$,  denoting 
$F^{(k)}$ the CDF for the uniform distribution on $\{2^{-\ell}/(1-2^{-k})\mid\ell\in\{1,\ldots,k\}\}$, then for any $x\in[0,1]$,
\[A_x=\begin{cases}\{\pi\in\Omega\mid F_\pi=F^{(k)}\} & \text{if }x=2^{-\ell}/(1-2^{-k}),\\
\emptyset & \text{otherwise}.\end{cases}\]
This follows from Theorem~\ref{prop:renewal_singular}(II). 

\begin{theorem}\label{prop:b111}
	Suppose $\{X_n\}_{n\ge1}$ conditioned on $\Pi$ is a binary process obtained from a stationary delayed renewal process as above. 
	Then we have the following for $F$ given by the mixture CDF in \eqref{e:mixtureF}.
	\begin{enumerate} 
		\item[{\rm (I)}] 
		For any $x\in [0,1]$, $F$ has a discontinuity at $x$ if and only if $x=2^{-\ell}/(1-2^{-k})$, where $k\in\mathbb N$, $\ell\in\{1,\ldots,k\}$, and $\mathrm P(F_\Pi=F^{(k)})>0$.
		\item[{\rm (II)}] For any $x\in [0,1]$, $F$ is strictly increasing at $x$ if and only if $\mathrm P(\Pi\in B_x)>0$. 
		\item[{\rm (III)}] 
		 For Lebesgue almost all $x\in[0,1]$, $F'(x)=\mathrm P(F_\Pi=F_1)$.
	\end{enumerate}
\end{theorem} 

\begin{example}[Mixture of special renewal processes]\label{ex:Renewalbeta} 
Let $\Pi$ be beta distributed 
and suppose that conditioned on $\Pi=\pi\in(0,1)$, we have a renewal process as considered in Example~\ref{ex:6}, i.e., $F_\pi$ is given by \eqref{eq:renewal_closed_form}. Then along similar lines as in Example~\ref{ex:MCbeta} but now applying Theorem~\ref{prop:b111}, we see that $F$ is singular continuous and strictly increasing  on $[0,1]$. However, evaluating the expected value of $F_\Pi$ leads to a complicated expression, which is omitted here. 
\end{example}

\section{Concluding remarks and open problems}\label{s:conclusion}
 
As pointed out in Section~\ref{s:intro}, singular continuous functions have been of much interest for many years and in \cite{part1} they appeared in connection to a decomposition result for the CDF of a random variable with stationary digits, considering different bases. In
Sections~\ref{s:main results1}--\ref{s:mixture-general} we provided general results and several examples of such singular continuous CDFs, where some examples were well-known and most were new. To obtain expressions of these CDFs, the finite dimensional distributions of the 
sequence of digits 
should be expressible in closed form as demonstrated in Examples~\ref{ex:3}--\ref{ex:MCbeta}.
Indeed many more examples could be given in the stationary case, 
but we leave this for future research.

In the binary case $q=2$ and when the digits are IID, the probability distribution of $X$ is related to Bernoulli convolutions, where in particular the absolute continuity or singularity and the Hausdorff dimension of the support of the distribution are of much interest, cf.\ Section~1.3 in \cite{part1} and the references therein. In fact, for any integer $q\ge2$, it is well-known that the probability measures corresponding to $X$'s with IID digits are mutually singular, cf.\ \cite{Billingsley}, and also the Hausdorff dimensions of the supports of such measures are known, see \cite{Yuval} and \cite{Varju}. To the best of our knowledge, a similar study
for other interesting model classes 
of the digits under stationarity remains to be done. Examples could be Markov chains (of any order), renewal processes, and mixtures of such models, where we have showed that for Markov chains and renewal processes $F$ is of pure type, whilst this is not necessarily the case if we consider mixtures of Markov chains or of renewal processes. 

 If we let 
$\{X_n\}_{n\in\mathbb Z}$ be a stochastic process with state space $\{0,\ldots,q-1\}$ and define
\[X_+\coloneqq\sum_{n=1}^\infty X_nq^{-n},\qquad X_-\coloneqq\sum_{n=0}^{\infty}X_{-n}q^{-n-1},\]
we may consider the bivariate CDF for $(X_+,X_-)$ which is concentrated on the unit square. What would be the general structure of this bivariate CDF when $\{X_n\}_{n\in\mathbb Z}$ is stationary, and which properties would the bivariate CDF possess under specific model classes for $\{X_n\}_{n\in\mathbb Z}$?
 Of course, the structure of the marginal CDFs of $X_+$ and $X_-$ are characterized by the results in \cite{part1}, and the examples of models considered in the present paper may easily be extended to stationary reversible processes $\{X_n\}_{n\in\mathbb Z}$. In particular,
  if $\ldots,X_{-1},X_{0},X_1,\ldots$ are IID, then $X_+$ and $X_-$ are independent. However, for other  
  cases $X_+$ and $X_-$ will in general be dependent.
We also defer such cases for future research.

It would also be interesting to study 
non-stationary models for the digits. For example, Minkowski’s question-mark function restricted to $[0,1]$ is a strictly increasing singular continuous CDF which does not satisfy the stationarity condition \eqref{e:3} for any integer $q\ge2$, cf.\ Corollary~2.17 in \cite{part1}. 
Further, in \cite{Wen1998} a Markov chain model (of order 1) for the random digits 
was considered without assuming stationarity and with $F$ being either the uniform CDF on $[0,1]$ or singular continuous (in agreement with Theorem~\ref{prop:markov_singular}). However, 
the assumptions in \cite{Wen1998} were more restrictive than in Section~\ref{s:MC-general} in assuming that all initial and transition probabilities are strictly positive, in contrast to our examples in Section~\ref{s:MC-general}, no `relatively closed formula' for 
the CDF was specified in \cite{Wen1998}. 

\section*{Acknowledgements}
This work was supported by The Danish Council for Independent Research | Natural Sciences, grant DFF – 10.46540/2032-00005B.

\section*{Appendix}

In this appendix we verify the theorems, the propositions, and the corollary in Sections~\ref{s:main results1}--\ref{s:mixture-general} which remain to be proven, and we establish some related results. It is convenient to introduce the notation 
\[(0.x_1\dots x_n)_q\coloneqq (0.x_1\dots x_n 0 0 \dots)_q\]
for $x_1,\dots,x_n\in \{0,\dots,q-1\}$.

\subsection*{A1 Proof of Proposition~\ref{t:useful}}\label{A:A.1}


Let $x=(0.x_1x_2\ldots)_q\in[0,1]$, where the non-terminating expansion is chosen if $x\in\mathbb Q_q$. Then $y=(0.y_1y_2\dots)_q$ is strictly less than $x$ if and only if $y_n<x_n$ for the first index $n$ where $y_n\neq x_n$, and thus
by the law of total probability, 
\begin{align*}
F(x)-\mathrm P\left(X=x\right)&=\mathrm P\left(\sum_{n=1}^\infty X_nq^{-n}<\sum_{n=1}^\infty x_nq^{-n}\right)\\
&=\sum_{n=1}^\infty\mathrm P(X_1=x_1,\ldots,X_{n-1}=x_{n-1},X_n<x_n)\\
&=\sum_{n\ge1:\,x_n\ge1}\sum_{k=0}^{x_n-1}p(x_1,\ldots,x_{n-1},k).
\end{align*}
Hence \eqref{e:F-id1} is verified and \eqref{e:F-id1} implies \eqref{e:F-id}.  


\subsection*{A2 Proof of Proposition~\ref{prop:useful:cont}}

 Let $x=(0.x_1x_2\ldots)_q\in[0,1]\setminus\mathbb Q_q$ (by stationarity, $F$ is continuous at $x$ if $x\in\mathbb Q_q$,
cf.\ Remark~\ref{rem:fractions}). Then, by monotonicity of probabilities, $$\mathrm P(X=x)=\lim_{n\rightarrow\infty}p(x_1,\ldots,x_n),$$
and so Proposition~\ref{prop:useful:cont} follows immediately.

\subsection*{A3 Proof of Proposition~\ref{prop:useful:derivative}}

The first part of Proposition~\ref{prop:useful:derivative} follows immediately from the mixture representation of $F$ as given by (I)--(III) in Section~\ref{s:setting}, where $c$ is the probability of obtaining the case (I) (i.e., $F=cF_1+(1-c)\tilde F$ where $\tilde F$ is a singular CDF on $[0,1]$).
 For the second part of Proposition~\ref{prop:useful:derivative} 
we only consider the case $m\geq 1$ as the case $m=0$ follows from similar arguments. Then the following lemma will be useful.


\begin{lemma}\label{lem:diff}
		Let $A\subset \mathbb{R}$ and suppose that $f\colon A\mapsto \mathbb{R}$ is differentiable at $x\in A$ and  
		$\{a_n\}_{n\ge1},\{b_n\}_{n\ge1}$ $\subset A\setminus\{x\}$ are sequences converging to $x$ such that there exists $c>0$ with $\vert a_n-b_n\vert \geq c\max\{\vert a_n-x\vert,\vert b_n-x\vert\}$ for all $n\in \mathbb{N}$. Then 
		\begin{align*}
		\lim_{n\to \infty} \frac{f(a_n)-f(b_n)}{a_n-b_n} =f'(x).
		\end{align*}
	\end{lemma}
	\begin{proof} 
		We have
		\begin{align*}
		&\Big\vert \frac{f(a_n)-f(b_n)}{a_n-b_n} -f'(x)\Big\vert 
		\le\\
		&\frac{\vert f(a_n)-f(x)-(a_n-x)f'(x)\vert+ \vert f(x)-f(b_n)-( x-b_n)f'(x)\vert}{\vert a_n-b_n\vert }
		\le\\
		 &\frac{1}{c} \Big\vert \frac{f(a_n)-f(x)}{a_n-x}-f'(x)\Big\vert+\frac{1}{c}\Big\vert \frac{f(x)-f(b_n)}{x-b_n}-f'(x)\Big\vert,
		\end{align*}
		where the first inequality follows from the triangle inequality and the second from	
		the assumption on $\vert a_n-b_n\vert$. 		
		 As the right hand side above goes to 0 for $n\to \infty$, the proof is complete.
	\end{proof}

 Note that in Lemma~\ref{lem:diff} it does not matter from which side the sequences $\{a_n\}_{n\geq 1}$ and $\{b_n\}_{n\geq 1}$ approach $x$. Thus, letting $n\in \mathbb{N}$ and $x$ be as in Proposition~\ref{prop:useful:derivative}, we define
 $a_n=(0.x_1\dots x_n \xi_1\dots \xi_m)_q+q^{-n-m}$ and $b_n=(0.x_1\dots x_n \xi_1\dots \xi_m)_q$. Observe that  $a_n,b_n\in[0,1]\setminus\{x\}$. Then
\begin{equation*}
	q^{-m}\max\{\vert a_n-x\vert,\vert b_n-x\vert\}\leq q^{-n-m} =\left\vert a_n-b_n\right\vert,
\end{equation*}
and hence by Lemma~\ref{lem:diff},
\[F'(x)=\lim_{n\to\infty} \frac{F(a_n)-F(b_n)}{a_n-b_n}.\]
Thus, since $F$ is continuous at base-$q$ fractions and both $a_n\in\mathbb Q_q$ and $b_n\in\mathbb Q_q$ for sufficiently large $n$, we get
\[F'(x)=\lim_{n\to\infty}q^{n+m}\left(\mathrm P(X<a_n)-\mathrm P(X<b_n)\right)=\lim_{n\to\infty} q^{n+m} p(x_1,\dots,x_m,\xi_1,\dots,\xi_m),
\]
whereby the proof is completed.

\subsection*{A4 Proof of Corollary~\ref{cor:useful:derivatives}}

We only prove Corollary~\ref{cor:useful:derivatives} for $m>1$, since the cases $m=0$ and $m=1$ follow from similar arguments. By Proposition~\ref{prop:useful:derivative}, 
\[\lim_{n\to\infty} q^{n+m-1}p(x_1,\dots,x_n,\xi_1,\dots,\xi_{m-1})=F'(x)>0.\]
This implies $p(x_1,\dots,x_n,\xi_1,\dots,\xi_m)>0$ for $n$ sufficiently large. Since also 
$$\lim_{n\rightarrow\infty}q^{n+m}p(x_1,\dots,x_n,\xi_1,\dots,\xi_m)=F'(x),$$ 
 we have 
\begin{equation*}
	\lim_{n\to\infty} \frac{p(x_1,\dots,x_n,\xi_1,\dots,\xi_m)}{p(x_1,\dots,x_n,\xi_1,\dots,\xi_{m-1})}=q^{-1}\lim_{n\to\infty}\frac{q^{n+m}p(x_1,\dots,x_n,\xi_1,\dots,\xi_m)}{q^{n+m-1}p(x_1,\dots,x_n,\xi_1,\dots,\xi_{m-1})}=q^{-1}.
\end{equation*}
Thereby \eqref{e:F'2} is verified. 

\subsection*{A5 Proof of Proposition~\ref{prop:useful:increasing}}

The proof is straightforward when considering each of the cases $x=0$, $x=1$, $x\in\mathbb Q_q$, and $x\in (0,1)\setminus\mathbb Q_q$. For instance, suppose $x=(0.x_1 x_2\dots)_q\in(0,1)\setminus\mathbb Q_q$. For $n\in \mathbb{N}$, define $y_n\coloneqq(0.x_1\dots x_n)_q$ and $z_n\coloneqq y_n+q^{-n}$. 
Then $y_n<x< z_n$ at least for sufficiently large $n$, and since $y_n\in\mathbb Q_q$ and $z_n\in\mathbb Q_q$, stationarity implies $\mathrm P(X=y_n)=\mathrm P(X=z_n)=0$, and so $F(z_n)-F(y_n)=p(x_1,\dots,x_n)$. Thereby Proposition~\ref{prop:useful:increasing} is verified in the case where $x=(0.x_1 x_2\dots)_q$ is a non-base-$q$ fraction in $(0,1)$.

\subsection*{A6 Proof of Theorem~\ref{prop:markov_singular}}

Assume $\{X_n\}_{n\ge1}$ is a stationary Markov chain of order $m-1$ for which $F$ is not the uniform CDF on $[0,1]$. Then it follows from \eqref{e:equil} that there must exist $\xi_1,\dots,\xi_{m}\in \{0,\dots,q-1\}$ such that $\pi_{\xi_1,\dots,\xi_{m}}\neq q^{-1}$. Let $x=(0.x_1x_2\dots)_q\in (0,1)\setminus\mathbb Q_q$ where $F'(x)$ exists. If $F'(x)>0$, then we obtain a contradiction: 
\begin{equation*}
	q^{-1}=\lim_{n\to\infty} \frac{p(x_1,\dots,x_n,\xi_1,\dots,\xi_m)}{p(x_1,\dots,x_n,\xi_1,\dots,\xi_{m-1})}= \pi_{\xi_1,\dots,\xi_m}\neq q^{-1},
\end{equation*}
where the first equality follows from Corollary~\ref{cor:useful:derivatives} and the second from the Markov property.
Consequently, $F'(x)=0$ and thus $F$ is singular.

	Using a notation as in Section~\ref{s:MC-general}, assume that $\pi_{x_1,\ldots,x_m}<1$ for all $x_1,\dots,x_m\in \{0,\dots, q-1\}$. 
	Consider any $x=(0.x_1x_2\dots)_q\in [0,1]$ and 
	define 
	\begin{align*}
	    \lambda\coloneqq\max_{b_1,\dots,b_m\in \{0,1,\dots,q-1\}}\pi_{b_1,\ldots,b_m} <1.
	\end{align*}
	Then, for any integer $n\geq m$, \eqref{e:Markovfinitedimdist} gives
\[p(x_1,\ldots,x_n)\le \lambda^{n-m+1}\rightarrow0\qquad\mbox{as $n\rightarrow\infty$.}\]
 Hence, Proposition~\ref{prop:useful:cont} gives that $F$ is continuous at $x$. This completes the proof of Theorem~\ref{prop:markov_singular}.	
 
\subsection*{A7 Proof of Proposition~\ref{prop-horiac}}\label{s:6.5.2}

For any $m\in\mathbb N$, denote the finite dimensional probabilities of $\{X_n^{(m)}\}_{n\geq 1}$ by
\[ p_m(x_1,\dots,x_n)=\mathrm P(X_1^{(m)}=x_1,\dots, X_n^{(m)}=x_n)\]
for $n\in \mathbb{N}$ and $x_1,\ldots,x_n\in\{0,\ldots,q-1\}$. 
By construction of $\{X_n^{(m)}\}_{n\geq 1}$, we have \(p_m(x_1,\dots,x_n)$ $=p(x_1,\dots,x_n)\)
whenever $n\le m$. For any $x=(0.x_1\ldots x_m)_q$, 
stationarity implies that $\mathrm P(X=x)=0$, and so 
it follows from \eqref{e:F-id1} that
\begin{equation}\label{eq:fmf3}
    F^{(m)}( (0.x_1\dots x_m)_q)=F( (0.x_1\dots x_m)_q).
\end{equation}
Let $x=(0.x_1x_2\dots)_q\in [0,1]$ be arbitrary. Combining \eqref{eq:fmf3} with the fact that $F^{(m)}$ is non-decreasing gives 
\[F((0.x_1\dots x_m)_q\leq F^{(m)}(x) \leq F((0.x_1\dots x_m)_q+q^{-m}).\]
Here, by the continuity of $F_3$, the left and the right hand side expressions of the inequalities converge to $F(x)$ as $m\rightarrow\infty$, 
so $F^{(m)}$ converges pointwise to $F$ (weak convergence). Hence, since $F$ is a continuous CDF, we obtain
\eqref{e:F3approx}, cf.\ \cite{10.2307/2237541}.

\subsection*{A8 Proof of Theorem~\ref{prop:renewal_singular}}\label{s:6.5.3}

 Let the situation be as in Theorem~\ref{prop:renewal_singular}. 
 
 The case (I) follows immediately since then the $X_n$'s are
  independent and uniformly distributed on $\{0,1\}$.

Suppose that $\mathrm P(Z_1=k)=1$ for some $k\in\mathbb N$. Then $Z_0$ is uniformly distributed on $\{1,\ldots,k\}$. If $\ell\in\{1,\ldots,k\}$ and $x=(0.x_1x_2\dots)_2=2^{-\ell}/(1-2^{-k})$, then  $x= \sum_{m=0}^\infty 2^{-\ell-km}$ and we get from \eqref{eq:renewal_probability} that 
$p(x_1,\dots,x_n)=1/k$. Thereby the case (II) is verified.

To show the case (III)  
assume first that $Z_1$ is not geometrically distributed with mean 2. Equivalently,  
\[ 
\mathrm P(Z_1>m\mid Z_1>m-1)\neq{1}/{2}
\] 
for some $m\in \mathbb{N}$. We show by contradiction that $F'(x)=0$ for all $x\in(0,1)\setminus\mathbb Q_2$ where $F'(x)$ exists. So
suppose that $F'(x)>0$ for some $x=(0.x_1x_2\dots)_2\in (0,1)\setminus\mathbb Q_2$. Let $\xi_1=1$ and $\xi_j=0$ for $j\in \{2,\dots,m+1\}$.
Then,  for any $n\in \mathbb{N}$, 
\begin{equation*}
	\frac{p(x_1,\dots,x_n,\xi_1,\dots,\xi_{m+1})}{p(x_1,\dots,x_n,\xi_1,\dots,\xi_m)}=\frac{\rmP(Z_1>m)}{\rmP(Z_1>m-1)}=\mathrm P(Z_1>m\mid Z_1>m-1)\neq{1}/{2}
\end{equation*}
using in the first identity \eqref{eq:renewal_probability} and that $Z_1,Z_2,\dots$ are identically distributed.
 This contradicts Corollary~\ref{cor:useful:derivatives} and thus $F'(x)=0$. In conclusion, $F$ is singular if $Z_1$ is not geometrically distributed with mean 2.
 
Suppose next that $\rmP(Z_1=k)<1$ for all $k\in \mathbb{N}$.
If $x\in [0,1]$ has finitely many digits equal to 1, then $\mathrm P(X=x)=0$. If $x$ has infinitely many digits equal to 1, then \eqref{eq:renewal_probability} gives
\[\lim_{n\rightarrow\infty}p(x_1,\ldots,x_n)\le\prod_{n:\ x_n=1} \sup_{m\in\mathbb N}\mathrm P(Z_1=m)=0,\]
since $\sup_{n\in\mathbb N}\mathrm P(Z_1=n)<1$. 
Hence, by Proposition~\ref{prop:useful:cont}, $F$ is continuous. 

Consequently, $F$ is singular continuous in case (III).

Finally, Theorem~\ref{prop:renewal_singular}(IV) follows immediately from Proposition~\ref{prop:useful:increasing} using \eqref{eq:renewal_probability} and the definition of the distribution of $Z_0$.

\subsection*{A9 Proof of Proposition~\ref{cor:G1G2}}

Let the situation be as in Proposition~\ref{cor:G1G2} and let $G$ be given by \eqref{eq:G}. By Proposition~\ref{prop:mixture-general}(V), there exists $c\in [0,1]$ such that 
$c=G'(x)=F'(x)-\rmP(F_\Pi=F_1)$ for  Lebesgue almost all $x\in [0,1]$. We will show that $c=0$ by contradiction, whereby 
Proposition~\ref{cor:G1G2} is verified. 

Suppose that $c>0$. 
For any $j\in \mathbb{N}$, define $U_j\coloneqq \{\pi\in \Omega\mid F_\pi\neq F_1\}\setminus V_j$. Since $V_j$ increases to $\{\pi\in \Omega \mid F_\pi\neq F_1\}$, there exists $j_0\in \mathbb{N}$ such that $P(\Pi\in U_{j_0})<c$. For $x\in \mathbb{R}$, define
\[J(x)\coloneqq\mathrm{E}(F_\Pi(x)1[\Pi\in U_{j_0}]),\qquad K(x)\coloneqq\mathrm{E}(F_\Pi(x)1[\Pi\in V_{j_0}]).\] 
Then $G=J+K$. Later on in this proof we will show that  $K$, if not identically zero, must at least be singular, which implies that for Lebesgue almost all $x\in [0,1]$, $J'(x)=G'(x)=c>0$, hence by the definition of $J$, we have $\rmP(\Pi\in U_{j_0})>0$. The function $J/\rmP(\Pi\in U_{j_0})$ is a CDF satisfying \eqref{e:3}, so by  Proposition~\ref{prop:mixture-general}(V), $c/\rmP(\Pi\in U_{j_0})=J'/ \rmP(\Pi\in U_{j_0})\leq 1$  almost everywhere on $[0,1]$. This is in contradiction with $\rmP(\Pi\in U_{j_0})<c$, 
	so $c=0$. 
	
Now let us show that $K'=0$ almost everywhere. Clearly, we may assume that  
 $\rmP(\Pi\in V_{j_0})>0$ so that $K$ is not identically zero. Then $K/\rmP(\Pi\in V_{j_0})$ is a CDF satisfying \eqref{e:3} and therefore it is differentiable almost everywhere. Let $A$ denote the set of points $x\in [0,1]$ for which both $K'(x)$ exists and the sequence $\{f_{x,n}(\cdot)1[\cdot\in V_{j_0}]\}_{n\geq 1}$ converges pointwise to $0$ as $n\to\infty$ and is dominated by a constant. Then $A$ has Lebesgue measure $1$, since, by assumption $A$ is the intersection of two sets of Lebesgue measure 1. Hence we can combine Proposition~\ref{prop:mixture-general}(V) with \eqref{e:mixture-p} and Lebesgue's dominated convergence theorem to obtain 
		\begin{equation*}
			K'(x)/\rmP(\Pi\in V_{j_0})=\lim_{n\to\infty} \mathrm{E}(q^n p_\Pi(x_1,\dots,x_n)1[\Pi\in V_{j_0}])/\rmP(\Pi\in V_{j_0})=0,
		\end{equation*}
		for all $x=(0.x_1x_2\dots)_q\in A$. Thus $K$ is singular.

\subsection*{A10 Proof of Theorem~\ref{prop:b1}}\label{a:proof-mix-MC}

Theorem~\ref{prop:b1}(I)--(II) follow directly from Proposition~\ref{prop:mixture-general}(III)--(IV). 

When verifying Theorem~\ref{prop:b1}(III) we assume $m\ge2$ and define without loss of generality $\Omega$ as follows (the case $m=1$ is simpler and follow similar lines as below). Denote  
the standard simplex in $\mathbb{R}^n$ by
\[\Delta_n\coloneqq \{(\alpha_1,\dots,\alpha_{n})\in [0,1]^n\mid \alpha_1+\dots+\alpha_{n}=1\}.\]
Consider the initial distribution $\{\pi(x_1,\dots,x_{m-1})\mid x_1,\ldots,x_{m-1}\in\{0,\ldots,q-1\}\}$ as a vector in $\Delta_{q^{m-1}}$ and the collection of transition probabilities $\{\pi_{x_1,\dots,x_m}\mid x_1,\ldots,x_m\in\{0,\ldots,q-1\}\}$ as a $q^{m-1}$-tuple of vectors in $\Delta_{q}$. Then $\Delta_{q^{m-1}}\times \Delta_{q}^{q^{m-1}}$ can be identified with the collection of distributions for Markov chains with state space $\{0,\ldots,q-1\}$ and of order $m-1$. Let $\Omega$ be the subset of those elements of $\Delta_{q^{m-1}}\times \Delta_{q}^{q^{m-1}}$ which correspond to stationary Markov chains of order $m-1$, and let the $\sigma$-algebra on $\Omega$ be induced by the Borel $\sigma$-algebra on $\mathbb{R}^{q^{m-1}+q^{m}}$.

We now verify the requirements of Proposition~\ref{cor:G1G2} whereby Theorem~\ref{prop:b1}(III) follows. Define 
\begin{equation*}
	\phi(\pi)\coloneqq q\prod_{t_1,\dots,t_{m}\in \{0,\ldots,q-1\}}\pi_{t_1,\dots,t_{m}}^{q^{-m}},\qquad \pi\in\Omega,
\end{equation*}
and let
\begin{equation}\label{eq:vn}
	V_j=\{\pi\in \Omega\mid \phi(\pi)\leq q^{-1/j}\},\qquad j\in\mathbb N.
\end{equation}
Then $V_1\subseteq V_2\subseteq\dots$.
Further, let $\pi^*\in\Omega$ correspond to having all initial probabilities equal to $q^{-m}$ and all transition probabilities equal to $q^{-1}$ (so $F_{\pi^*}$ is the uniform CDF on $[0,1]$). Then $\pi^*$ is the unique maximizer of $\phi$ and $\phi(\pi^*)=1$. This follows e.g.\ with the use of Lagrange multipliers. Furthermore, it follows from \eqref{e:equil} that $F_\pi=F_1$ if and only if $\pi=\pi^*$. Hence
\[\bigcup_{j=1}^\infty V_j=\Omega\setminus\{\pi^*\}=\{\pi\in \Omega\mid F_\pi\neq F_1\},\]
which is one requirement of Proposition~\ref{cor:G1G2}.

To verify the other requirement of Proposition~\ref{cor:G1G2} we need some notation.
For any $x=(0.x_1x_2\ldots)_q\in[0,1]\setminus\mathbb Q_q$ and any $n\in\mathbb N$, define 
\begin{equation*}
		\phi_{x,n}(\pi)\coloneqq q\prod_{t_1,\dots,t_{m}\in \{0,\ldots,q-1\}}\pi_{t_1,\dots,t_{m}}^{n_{t_1,\dots,t_{m}}(x_1,\dots,x_n)/n},\qquad \pi\in\Omega,
	\end{equation*}
where $n_{t_1,\dots,t_j}(x_1,\dots,x_k)$ is the number of times the string $t_1,\dots,t_j$ appears in the string $x_1,\dots,x_k$.	
Then the other requirement of Proposition~\ref{cor:G1G2}
states that for all $j\in\mathbb N$, Lebesgue
almost all $x=(0.x_1x_2\ldots)_q\in[0,1]\setminus\mathbb Q_q$, and all 
$\pi\in\Omega$ we have that
\begin{align*}
&q^n \pi(x_1,\ldots,x_{m-1})\prod_{j=m}^n \pi_{x_{j-m+1},\ldots,x_j} 1[\phi(\pi)\le q^{-1/j}]\\
= &\, \pi(x_1,\ldots,x_{m-1}) \phi_{x,n}(\pi)^n1[\phi(\pi)\le q^{-1/j}]
\end{align*}
converges to 0 as $n\rightarrow\infty$ and is less than some number $c(x,j)$. 
To verify this   
we recall that $(0.x_1x_2\dots)_q\in [0,1]$ is a \emph{normal number} (in base $q$) if for all $j\in\mathbb N$ and all $t_1,\dots,t_j\in\{0,\ldots,q-1\}$,
\begin{align}\label{eq:b3}
\lim_{k\to \infty} {n_{t_1,\dots,t_j}(x_1,\dots,x_k)}/{k}=q^{-j}.
\end{align} 
Since the set of normal numbers in $[0,1]$ has Lebesgue measure $1$~\cite{Borel1909}, we can assume that $x$ is normal. By \eqref{eq:b3} there exists some $n_j\in\mathbb N$ such that for all $t_1,\dots,t_m\in\{0,\dots,q-1\}$,
\[n_{t_1,\dots,t_m}(x_1,\dots,x_n)/n\geq \frac{2j}{(2j+1)q^m}\]
whenever $n\geq n_j$. Consequently, for all $n\geq n_j$ and $\pi\in V_j$, 
\begin{align*}	
	\phi_{x,n}(\pi)\leq q\prod_{t_1,\dots,t_{m}\in \{0,\ldots,q-1\}}\pi_{t_1,\dots,t_{m}}^{\frac{2j}{(2j+1)q^m}}= (q\phi(\pi)^{2j})^{1/(2j+1)}\leq q^{-1/(2j+1)}<1,
\end{align*}
where the second last inequality uses \eqref{eq:vn}. Thereby the other requirement of Proposition~\ref{cor:G1G2} follows. 

\subsection*{A11 Proof of Theorem~\ref{prop:b111}}\label{a:proofren}	

Theorem~\ref{prop:b111}(I)--(II) follow directly from Proposition~\ref{prop:mixture-general}(III)--(IV). Below we verify the requirements of Proposition~\ref{cor:G1G2} whereby Theorem~\ref{prop:b111}(III) follows.

Define without loss of generality 
\[\Omega=\left\{(\pi_1,\pi_2,\dots)\in \ell^1(\mathbb{N})\mid \sum_{k=1}^\infty \pi_k=1, \textup{ and }  \sum_{k=1}^\infty k\pi_k<\infty\right\}\]
and let the $\sigma$-algebra on $\Omega$ be induced by the Borel $\sigma$-algebra on $\ell^1(\mathbb{N})$ (the space of absolutely summable sequences) equipped with the usual $\ell^1$-norm. 
For any $m\in \mathbb{N}$ and $\pi\in\Omega$, define 
\begin{equation*}
	\phi_m(\pi)\coloneqq 2\prod_{k=1}^m\pi_k^{2^{-k-1}},
	\qquad \pi^{(m)}\coloneqq (1-2^{-m})^{-1} \left( 2^{-1},\dots,2^{-m},0,0,\dots \right),
\end{equation*}
and	
\[\phi(\pi)\coloneqq 2\prod_{k=1}^\infty\pi_k^{2^{-k-1}},
	\qquad \pi^*\coloneqq(1/2,1/4,1/8,\ldots)	,
\]
so $\pi^*\in\Omega$ is the geometric distribution with mean $2$. We now verify that
\begin{equation}\label{e:uniquemax}
\mbox{$\pi^*$ is the unique maximizer of $\phi$.}
\end{equation}
It is easily seen that $\pi^{(m)}\in\Omega$ is the unique maximizer of $\phi_m$, and since for all $\pi\in\Omega$, $\phi_m(\pi)$ non-increases towards $\phi(\pi)$ as $m\rightarrow\infty$, it follows that 
\[\phi(\pi)\le\lim_{m\rightarrow\infty}\phi_m\left(\pi^{(m)}\right)=\lim_{m\rightarrow\infty}2^{(m+2)/2^{m+1}}(1-2^{-m})^{(2^{-m}-1)/2}=1.\] 
Furthermore, $\Omega$ is a convex set and we claim that $\phi$ is a strictly log-concave function: Consider any $\alpha\in(0,1)$ and distinct $\pi,\tilde\pi\in\Omega$, so there exist some $\pi_\ell\not=\tilde\pi_\ell$. Since $\ln$ is strictly concave, we have $\ln(\alpha\pi_\ell+(1-\alpha)\tilde\pi_\ell)>\alpha \ln\pi_\ell+(1-\alpha)\ln\tilde\pi_\ell$ and $\ln(\alpha\pi_k+(1-\alpha)\tilde\pi_k)\ge\alpha \ln\pi_k+(1-\alpha)\ln\tilde\pi_k$ for $k\not=l$, so
\[\ln\phi(\alpha\pi+(1-\alpha)\tilde\pi)>\alpha\ln\phi(\pi)+ (1-\alpha)\phi(\tilde\pi).\] 
Consequently,  $\phi\le1$ is strictly log-concave 
 with
\[\phi\left(\pi^*\right)=2^{1-\sum_{k=1}^\infty k2^{-k-1}}=1,\]
and so \eqref{e:uniquemax} follows.

Now, for any $j\in \mathbb{N}$, define
\begin{equation}\label{eq:vj1}
	V_j\coloneqq\{\pi\in \Omega\mid \phi_m(\pi)\leq 2^{-1/m} \textup{ whenever } m\geq j\}.
\end{equation}
Then $V_1\subseteq V_2\subseteq\dots$. For all $m\in \mathbb{N}$,
 a direct calculation gives $\phi_m(\pi^*)=2^{(m+2)/2^{m+1}}>1$, so $\pi^*\notin \bigcup_{j=1}^\infty V_j$. On the other hand, 
suppose $\pi\in \Omega\setminus\pi\notin \bigcup_{j=1}^\infty V_j$. Then by \eqref{eq:vj1} we can find an increasing sequence $\{m_k\}_{k\geq 1}$ such that for all $k\in \mathbb{N}$, we have
$\phi_{m_k}(\pi)>2^{-1/m_k}$, and so taking the limit as $k\rightarrow$ we obtain $\phi(\pi)\ge1$. Therefore, by \eqref{e:uniquemax}, $\pi=\pi^*$, so 
\[\bigcup_{j=1}^\infty V_j=\Omega\setminus\{\pi^*\}.\]
From Theorem~\ref{prop:renewal_singular}(I) and the definition of $\pi^*$ we obtain that
$F_\pi=F_1$ if and only if $\pi=\pi^*$, and hence
\[\bigcup_{j=1}^\infty V_j=\{\pi\in \Omega\mid F_\pi\neq F_1\},\]
which is one requirement of Proposition~\ref{cor:G1G2}. 

For any $j\in \mathbb{N}$ and any  normal number $x=(0.x_1x_2\dots)_q\in [0,1]$ (cf.\ \eqref{eq:b3}), define
\[\phi_{x,n}(\pi)\coloneqq 2\prod_{k=1}^\infty\pi_k^{n_k(x_1,\ldots,x_n)/n},\qquad \pi\in \Omega.\] 
By \eqref{eq:b3} there exists some $n_j\in\mathbb N$ such that for all $k\in \{1,\dots,j\}$,
\begin{equation}\label{eq:nk}
	n_{k}(x_1,\dots,x_n)/n\geq \frac{2j}{(2j+1)2^{k+1}}
\end{equation}
whenever $n\geq n_j$. Hence, for all $n\geq n_j$ and all $\pi\in V_j$, 
\begin{align*}	
		\phi_{x,n}(\pi)\leq 2\prod_{k=1}^j\pi_{k}^{n_k(x_1,\dots,x_n)/n}\leq 2\prod_{k=1}^j\pi_{k}^{\frac{2j}{(2j+1)2^{k+1}}}= (2\phi_j(\pi)^{2j})^{1/(2j+1)}\leq 2^{-1/(2j+1)},
\end{align*}
where the first inequality uses that $\prod_{k=1}^m \pi_{k}^{n_k(x_1,\dots,x_n)/n}\downarrow \prod_{k=1}^\infty \pi_{k}^{n_k(x_1,\dots,x_n)/n}$ as $m\to\infty$, the second inequality uses \eqref{eq:nk}, and the last inequality uses \eqref{eq:vj1}. Therefore, for all $n\geq n_j$ and all $\pi\in V_j$, we have $\phi_{x,n}(\pi)<1$. Furthermore, by definition, 
for all $\pi\in V_j$ and $n\in \mathbb{N}$, we have
$\phi_{x,n}(\pi)^n\leq 2^n$. So, for any $\pi\in\Omega$, we get that $\{\phi_{x,n}^n1[\pi\in V_j]\}_{n\geq 1}$ is dominated by the constant $2^{n_j}$ and converges pointwise to $0$ as $n\to\infty$. Thereby the other requirement of Proposition~\ref{cor:G1G2} holds, and so the proof of Theorem~\ref{prop:b111} is completed.

\bibliographystyle{spbasic}      
\bibliography{bibliography}   

\end{document}